\documentclass{amsart}

\usepackage{mathtools,amsthm}
\usepackage[letterpaper, margin=1.3in]{geometry}
\usepackage[utf8]{inputenc}
\usepackage[T2A,T1]{fontenc}
\usepackage{hyperref}
\usepackage{enumitem}
\usepackage{mathrsfs}
\usepackage{tikz-cd}
\usepackage{amssymb}
\usepackage{xspace}
\usepackage{nicefrac}
\usepackage{multicol}
\usepackage{etoolbox}
\usepackage[nameinlink]{cleveref}
\usepackage[titletoc]{appendix}
\usepackage{setspace}
\usepackage{bbm}
\usepackage{versions}
\usepackage{stmaryrd}

\DeclareFontFamily{U}{wncy}{}
\DeclareFontShape{U}{wncy}{m}{n}{<->wncyr10}{}
\DeclareSymbolFont{mcy}{U}{wncy}{m}{n}
\DeclareMathSymbol{\Sh}{\mathord}{mcy}{"58}
\newcommand{\Sha}{\Sh}

\newtheorem{theorem}{Theorem}[section]
\newtheorem*{theorem*}{Theorem}
\newtheorem{lemma}[theorem]{Lemma}
\newtheorem{corollary}[theorem]{Corollary}
\newtheorem{proposition}[theorem]{Proposition}

\newtheorem*{conjecture*}{Conjecture}
\crefname{theorem}{Theorem}{Theorems}
\crefname{lemma}{Lemma}{Lemmas}
\crefname{corollary}{Corollary}{Corrolaries}
\crefname{proposition}{Proposition}{Propositions}
\crefname{question}{Question}{questions}
\crefname{conjecture}{Conjecture}{Conjectures}

\newtheorem{innercustomgeneric}{\customgenericname}
\providecommand{\customgenericname}{}
\newcommand{\newcustomtheorem}[2]{%
  \newenvironment{#1}[1]
  {%
   \renewcommand\customgenericname{#2}%
   \renewcommand\theinnercustomgeneric{\kern-0.3em ##1}%
   \innercustomgeneric
  }
  {\endinnercustomgeneric}
}

\newcustomtheorem{specialtheorem}{}
\crefname{specialtheorem}{}{}

\theoremstyle{definition}
\newtheorem{definition}[theorem]{Definition}

\newtheorem{assumption}[theorem]{Assumption}
\crefname{definition}{Definition}{Definitions}
\crefname{hypothesis}{Hypothesis}{Hypothesis}
\crefname{assumption}{Assumption}{Assumptions}

\theoremstyle{remark}
\newtheorem{remark}[theorem]{Remark}
\crefname{remark}{Remark}{Remarks}

\crefname{appsec}{Appendix}{Appendices}
\renewenvironment{appendix}{\newpage\begin{appendices}\crefalias{subsection}{appsec}}{\end{appendices}}

\numberwithin{theorem}{subsection}
\numberwithin{equation}{subsection}

\newcommand{\xtwoheadrightarrow}[2][]{\xrightarrow[#1]{#2}\mathrel{\mkern-14mu}\rightarrow
}
\newcommand{\abs}[1]{\##1}
\newcommand{\Gal}[2]{\mathrm{Gal}\left(#1\slash #2\right)}

\newcommand{\Z}{\mathbb{Z}}
\newcommand{\Q}{\mathbb{Q}}

\renewcommand{\C}{\mathbb{C}}
\newcommand{\F}{\mathbb{F}}
\newcommand{\p}{\mathfrak{p}}
\newcommand{\m}{\mathfrak{m}}

\renewcommand{\P}{\mathfrak{P}}
\renewcommand{\O}{\mathcal{O}}
\renewcommand{\L}{\mathscr{L}}
\renewcommand{\H}{\mathcal{H}}
\newcommand{\T}{\mathbb{T}}
\newcommand{\A}{\mathbb{A}}

\newcommand{\Tr}{\mathrm{Tr}}
\newcommand{\Hom}{\mathrm{Hom}}
\newcommand{\rightiso}{\xrightarrow{\sim}}
\newcommand{\iso}{\simeq}
\newcommand{\xrightiso}[1]{\xrightarrow[\sim]{#1}}

\newcommand{\ord}[1]{\mathrm{ord}_{#1}}
\newcommand{\units}[1]{{#1}^{\times}}
\renewcommand{\exp}{\mathrm{exp}}
\newcommand{\Ker}[1]{\mathrm{Ker}\left(#1\right)}
\renewcommand{\Im}[1]{\mathrm{Im}\left(#1\right)}
\newcommand{\Coker}[1]{\mathrm{Coker}\left(#1\right)}
\newcommand{\conj}[1]{\overline{#1}}
\newcommand{\Index}[2]{\left[#1:#2\right]}
\newcommand{\Cohomology}[4]{\mathrm{H}^{#1}_{{#2}}\left(#3,#4\right)}
\newcommand{\loc}[1]{\mathrm{loc}_{#1}}
\newcommand{\rel}{\mathrm{rel}}
\newcommand{\str}{\mathrm{str}}
\renewcommand{\f}{\mathrm{f}}
\newcommand{\s}{\mathrm{s}}
\newcommand{\tr}{\mathrm{tr}}
\newcommand{\ur}{\mathrm{ur}}
\renewcommand{\v}{\mathfrak{v}}
\newcommand{\w}{\mathfrak{w}}
\renewcommand{\log}{\mathrm{log}}
\newcommand{\rank}{\mathrm{rank}}
\newcommand{\Char}{\mathrm{char}}
\newcommand{\length}[2][]{\mathrm{length}_{#1}\left(#2\right)}
\renewcommand{\div}{\mathrm{div}}
\newcommand{\tor}{\mathrm{tor}}
\newcommand{\defeq}{\vcentcolon=}
\newcommand{\eqdef}{=\vcentcolon}
\newcommand{\isequal}{\stackrel{?}{=}}

\newcommand{\equal}{\sim_p}
\newcommand{\ac}{\mathrm{ac}}

\newcommand{\BK}{\mathcal{F}_{\mathrm{BK}}}
\newcommand{\KS}{\mathbf{KS}}
\newcommand{\FLambda}{\mathcal{F}_\Lambda}
\newcommand{\Fac}{\mathcal{F}_{\mathrm{ac}}}
\newcommand{\dlim}[1]{\varinjlim\limits_{#1}}
\newcommand{\ilim}[1]{\varprojlim\limits_{#1}}

\newcommand{\dperiod}[3]{\mathbb{D}_{\mathrm{#1}\def\temp{#2}\ifx\temp\empty\else,\fi#2}\def\temp{#3}\ifx\temp\empty\else(#3)\fi}
\newcommand{\dcris}[2]{\dperiod{cris}{#1}{#2}}
\newcommand{\ddr}[2]{\dperiod{dR}{#1}{#2}}

\newcommand{\Fil}[2]{\mathrm{Fil}^{#1}\left(#2\right)}

\newcommand{\dr}[2]{\ddr{#1}{#2}}

\newcommand{\pair}[2]{\left<#1,\ #2\right>}

\newcommand{\grpring}[1]{\llbracket#1\rrbracket}

\newcommand{\eqnumbersubsection}{\renewcommand{\theequation}{\thesubsection.\alph{equation}}\numberwithin{theorem}{subsection}}

\let\oldequation\equation
\let\endoldequation\endequation
\newcommand{\tmpArg}{}
\newenvironment{namedequation}{\begin{equation*}}{\end{equation*}}

\renewenvironment{equation}[1][]
{%
    \ifstrempty{#1}{%
        \renewcommand{\tmpArg}{\endnamedequation}%
        \namedequation%
    }{%
        \renewcommand{\tmpArg}{\endoldequation}%
        \oldequation\label{#1}%
    }%
}%
{%
\tmpArg\ignorespacesafterend%
}

\patchcmd{\section}{\normalfont}{\normalfont\LARGE}{}{}

\begin{document}

\eqnumbersubsection

\newcommand{\nocontentsline}[3]{}
\newcommand{\tocless}[2]{\bgroup\let\addcontentsline=\nocontentsline#1{#2}\egroup}

\title{On Howard's main conjecture and the Heegner point Kolyvagin system}
\author{Murilo Zanarella}

\begin{abstract}
We upgrade Howard's divisibility toward Perrin-Riou's Heegner point Main Conjecture to an equality under some mild conditions. We do this by exploiting Wei Zhang's proof of the Kolyvagin conjecture. The main ingredient is an improvement of Howard's Kolyvagin system formalism. As another consequence of it, we establish the equivalence between this main conjecture and the primitivity of the Kolyvagin system in certain cases, by also exploiting a explicit reciprocity law for Heegner points.
\end{abstract}

\date{26 August, 2019}

\maketitle

\setcounter{tocdepth}{3}
\let\oldtocsection=\tocsection
\let\oldtocsubsection=\tocsubsection
\let\oldtocsubsubsection=\tocsubsubsection
\renewcommand{\tocsection}[2]{\hspace{0em}\oldtocsection{#1}{#2}}
\renewcommand{\tocsubsection}[2]{\hspace{1em}\oldtocsubsection{#1}{#2}}
\renewcommand{\tocsubsubsection}[2]{\hspace{2em}\oldtocsubsubsection{#1}{#2}}
\vspace{-0.5in}
{\singlespacing
\tableofcontents}

\section{Introduction}
Fix a quadratic imaginary field $K,$ a prime $p\ge5$ and an elliptic curve $E\slash\Q$ with good reduction at $p$ and with conductor $N_E$ coprime with $D_K.$ Let $T=T_pE$ be the Tate module of $E.$ For a number field $L,$ let $\mathrm{Sel}_{p^\infty}(E/L)$ and $S_p(E/L)$ denote respectively the usual discrete and compact $p$-adic Selmer groups, which fit in the fundamental exact sequences
\begin{equation}
    0\to E(L)\otimes\Q_p/\Z_p\to\mathrm{Sel}_{p^\infty}(E/L)\to\Sha(E/L)[p^\infty]\to0,
\end{equation}
\begin{equation}
    0\to E(L)\otimes\Z_p\to S_p(E/L)\to\ilim{n}\Sha(E/L)[p^n]\to0.
\end{equation}

Let $K_\infty$ denote the anticyclotomic $\Z_p$-extension of $K,$ which is the unique $\Z_p$-extension of $K$ such that the nontrivial element of $\Gal{K}{\Q}$ acts by $-1$ via conjugation, and denote by $\Psi\colon G_K\twoheadrightarrow\Gal{K_\infty}{K}$ the natural projection. We consider the Iwasawa algebra $\Lambda\defeq \Z_p\grpring{\Gal{K_\infty}{K}}$ and the $\Lambda$-modules $\T\defeq T\otimes_{\Z_p}\Lambda$ and $\A\defeq T\otimes_{\Z_p}\Lambda^\vee$ with diagonal Galois action by $\Psi$ on $\Lambda$ and $\Psi^{-1}$ on $\Lambda^\vee,$ where $\Lambda^\vee\defeq\Hom_{\mathrm{cont}}(\Lambda,\Z_p).$

Assuming first that $p$ has ordinary reduction, we consider the generalized Selmer groups
\begin{equation}
    \Cohomology{1}{\FLambda}{K}{\T}\subseteq\ilim{K\subseteq L\subset K_\infty}\Cohomology{1}{}{L}{T},\quad \Cohomology{1}{\FLambda}{K}{\A}\subseteq\dlim{K\subseteq L\subset K_\infty}\Cohomology{1}{}{L}{E[p^\infty]},
\end{equation}
defined in \cite[Definition 3.2.2]{Howard2}. These are such that there are pseudo-isomorphisms
\begin{equation}
    \Cohomology{1}{\FLambda}{K}{\T}\sim \ilim{K\subseteq L\subset K_\infty}S_p(E/L),\quad \Cohomology{1}{\FLambda}{K}{\A}\sim\dlim{K\subseteq L\subset K_\infty}\mathrm{Sel}_{p^\infty}(E/L).
\end{equation}
In this setting, the anticyclotomic Iwasawa Main Conjecture for $E,$ first formulated by Perrin-Riou in \cite{Perrin-Riou} and later extended by Howard in \cite{Howard2}, predicts a relation between a cohomology class $\kappa_1^{\mathrm{Hg}}\in\Cohomology{1}{\FLambda}{K}{\T}$ constructed from the Euler system of Heegner points and
\begin{equation}
    X\defeq\Hom_{\Z_p}(\Cohomology{1}{\FLambda}{K}{\A},\Q_p/\Z_p).
\end{equation}
More precisely, it conjectures that
\begin{equation}
    \Char(X_\tor)\isequal\Char\left(\Cohomology{1}{\FLambda}{K}{\T}/\kappa_1^{\mathrm{Hg}}\right)^2,
\end{equation}
where $\Char(M)$ denotes the characteristic ideal of a torsion $\Lambda$-module.

In the case where $E$ has supersingular reduction at $p,$ an analogue of such conjecture was considered by Castella--Wan in \cite{Castella-Wan} in the case where $p$ splits in $K.$ This is obtained by considering a Selmer structure with different local conditions above $p,$ and by considering a modified version of $\kappa_1^{\mathrm{Hg}}.$

Recently, such conjectures have been verified in certain cases. The divisibility
\begin{equation}
   \Char(X_\tor)\mid \Char(\Cohomology{1}{\FLambda}{K}{\T}/\kappa_1^{\mathrm{Hg}})^2
\end{equation}
was obtained in \cite[Theorem 3.4.3]{Howard2} (and was extended to the supersingular case in \cite[Theorem 5.12]{Castella-Wan}). The opposite divisibility can be obtained in certain cases by exploiting the relation with the BDP main conjecture, whose corresponding divisibility can be obtained by the results of \cite[Theorem 1.2]{Wan} on a two variable main conjecture, as done in \cite[Theorem 1.2]{Wan2}. Wan also proved the equality up to powers of $p$ in more generality, and this can be upgraded to an equality by exploiting the vanishing of the $\mu$-invariant of the BDP $p$-adic $L$-function (\cite[Theorem B]{Burungale} and \cite[Theorem B]{Hsieh}), as done in \cite[Theorem 28]{Skinner} and \cite[Theorem 6.1]{Castella-Wan}.

The goal of this paper is two-fold. First, we build on the work of Howard to show how the full Howard's Main Conjecture can be established from the primitivity of the usual Kolyvagin system (\cref{Theorem-A}), by adapting arguments of \cite{Mazur-Rubin} to the setting of \cite{Howard}. Such primitivity was proven in greater generality by Wei Zhang \cite{Wei-Zhang}, and this will allow us to obtain new cases of Howard's Main Conjecture (\cref{newresult}). We note that our approach is of a different nature than the recent results mentioned above, and that, to the author's knowledge, this is the first instance where one could deduce a main conjecture from primitivity conditions.

Secondly, we show how we can use our improvement in the formalism of Howard to determine, at least conjeturally, when the Kolyvagin system of Heegner points should be primitive. Kolyvagin observed in \cite[Theorem E]{Kolyvagin}, under mild hypothesis and for elliptic curves $E\slash \Q$ with analytic rank $1$ over $K,$ that, conditional on the $p$ part of the BSD formula, primitivity should happen exactly when certain Tamagawa factors are $p$-indivisible. We'll prove an analog of this for higher analytic rank, but conditional on Howard's Main Conjecture instead (\cref{Theorem-B}). This will be done by considering twists of the usual Kolyvagin system by finite order anticyclotomic characters. By the work of Cornut--Vatsal in \cite{Cornut-Vatsal}, we can choose such twist in a way that the Kolyvagin system has nonzero base class, and hence exploit a explicit reciprocity law for Heegner points to relate the primitivity of the Kolyvagin system with a special value of the BDP $p$-adic $L$-function.

The two results above actually amount to the equivalence between Howard's Main Conjecture and primitivity in certain cases (\cref{Theorem-C}).

\subsection{Main results}
Let $E/\Q$ be an elliptic curve and $K$ be an imaginary quadratic field with discriminant $D_K<-4.$ Let $N=N_E$ be the conductor of $E\slash\Q.$ We assume that $(N,D_K)=1.$ Let $N=N^+N^-$ be the factorization such that primes $l\mid N^+$ are split in $K$ and primes $l\mid N^-$ are inert in $K.$ We assume that $K$ satisfies the generalized Heegner Hypothesis
\begin{equation}[Heeg]\tag{Heeg}
    N^-\text{ is a square-free product of an even number of primes.}
\end{equation}

Let $p$ be a prime number $p\nmid D_K$ such that
\begin{equation}[good]\tag{good}
    E\text{ has good reduction at }p,
\end{equation}
\begin{equation}[p-big]\tag{$p$ big}
    p\ge 5,
\end{equation}
and denote by $T=T_pE$ the Tate module. We assume that
\begin{equation}[res-surj]\tag{res-surj}
    \text{the residual representation of }T\text{ is surjective.}
\end{equation}
As shown in \cite[Th\'eor\`em 2]{Serre}, this last condition is true for all but finitely many primes $p$ if $E$ does not have complex multiplication.

For certain arguments we will also need to assume that
\begin{equation}[split]\tag{split}
    p\text{ splits in }K,
\end{equation}
and that $p$ is not anomalous, that is,
\begin{equation}[ap-1]\tag{not anom}
    p\nmid \abs{\tilde{E}(\F_v)}\text{ for a place }v\mid p\text{ of }K.
\end{equation}
Note that \eqref{ap-1} is equivalent to
\begin{equation}
    \left\{\begin{array}{ll} p\nmid a_p-1&\text{if }p\text{ is split in }K,\\p\nmid a_p^2-1&\text{if }p\text{ is inert in }K,\end{array}\right.
\end{equation}
where $a_p$ is the usual trace of Frobenius $a_p\defeq p+1-\abs{\tilde{E}(\F_p)}.$

Let $K_\infty$ be the $\Z_p$-anticyclotomic extension of $K.$ Let $\Lambda\defeq\Z_p\grpring{\Gal{K_\infty}{K}}$ and denote by $\m=(p,\gamma-1)$ its maximal ideal, where $\gamma\in\Gal{K_\infty}{K}$ is any topological generator. We consider the $\Lambda$-modules $\T=T\otimes\Lambda,$ $\A=T\otimes\Lambda^{\vee}$ with Galois action on both factors, and $\Lambda$ action on the second factor.

If $M=T$ or $\T$ or $\A$ or a quotient of $T$ or $\T,$ and $\mathcal{F}=(\mathcal{F}_v)_v$ is a collection of submodules $\mathcal{F}_v\subseteq\Cohomology{1}{}{K_v}{M}$ for every place $v$ of $K,$ we denote
\begin{equation}
    \Cohomology{1}{\mathcal{F}}{K}{M}\defeq\Ker{\Cohomology{1}{}{K}{M}\to\bigoplus_v\frac{\Cohomology{1}{}{K_v}{M}}{\mathcal{F}_v}}
\end{equation}
to be the corresponding Selmer group.

We let $\BK$ denote the Bloch--Kato Selmer structure\footnote{In general, if $V$ is a finite dimensional $G_F$-representation over $\Q_p$ for some number field $F,$ the Bloch--Kato local conditions are unramified away from $p$ and $\Ker{\Cohomology{1}{}{F_v}{V}\to\Cohomology{1}{}{F_v}{V\otimes_{\Q_p}B_{\mathrm{cris}}}}$ for $v\mid p,$ and these local conditions are propagated to quotients and submodules of $V.$} for $T,$ whose local conditions in this case the exact annihilators under Tate duality of the images of the local Kummer maps. We also propagate these conditions to quotients of $T.$ Denote by $\kappa$ the usual Kolyvagin system of Heegner points for $\Cohomology{1}{\BK}{K}{T}.$ This is a collection of classes (see \Cref{kolyvagin-systems-section} for the precise definition)
\begin{equation}
    \kappa=\{\kappa_n\in\Cohomology{1}{\BK(n)}{K}{T/I_nT}\colon n\text{ square-free product of Kolyvagin primes}\}
\end{equation}
that can be constructed under the above assumptions together with the following \cref{assumption-bk}.

For a prime $w$ of $K,$ let $\F_w$ denote its residue field, and let $\mathcal{E}$ be the N\'eron model of $E$ over $K_w.$ Let $\pi_0(\mathcal{E}_0)$ be the group of connected components of the special fiber $\mathcal{E}_0$ of $\mathcal{E},$ and recall that it has a $\Gal{\overline{\F_w}}{\F_w}$ action. We denote by $c_w(E/K)\defeq \abs{\pi_0(\mathcal{E}_0)(\F_w)}$ the Tamagawa factor of $E/K$ at $w.$
\begin{assumption}\label{assumption-bk}
    Either $N^-=1$ or $p\nmid c_w(E/\Q)$ for all $w\mid N^+.$
\end{assumption}

\begin{remark}
    This assumption will only be used to ensure that the cohomology classes $\kappa_n\in\Cohomology{1}{}{K}{T/I_nT}$ and its twists lie in the $\BK(n)$ Selmer group, and can likely be removed, as explained in \cref{assumption-bk-remark}. We also note that such assumption is already implied by the hypothesis of \cite[Theorem 1.1]{Wei-Zhang}.
\end{remark}

We say that $\kappa$ is \emph{primitive} if its reduction modulo $p,$ denoted $\overline{\kappa}=\{\overline{\kappa}_n\in\Cohomology{1}{\BK(n)}{K}{T/pT}\},$ is nonzero.

We consider the Selmer group $\Cohomology{1}{\FLambda}{K}{\T}$ mentioned in the introduction. At primes $v\nmid p,$ we have $(\FLambda)_v=\Cohomology{1}{}{K_v}{\T}.$ For primes $v$ above $p,$ $(\FLambda)_v$ depends on the type of reduction of $p.$ For $p$ ordinary, this is the ordinary submodule as in \cite[Definition 3.2.2]{Howard2}. For $p$ supersingular, this is the $+$ submodule of \cite[Section 3.3]{Castella-Wan}. In the supersingular case, one could get different Selmer modules by choosing such local condition to be either $+$ or $-.$ It will be important that we work with the $+$ condition\footnote{Otherwise the proof of \cref{lambdaprimitive-ss} would not work.}. Now the Selmer group $\Cohomology{1}{\FLambda}{K}{\A}$ is defined to have local conditions to be the exact orthogonal complements of the above ones under Tate local duality.

By the work of Howard in the ordinary case and of Castella--Wan in the supersingular case, we know that there exist a $\Lambda$-adic Kolyvagin system $\kappa^{\mathrm{Hg}}$ for $\Cohomology{1}{\FLambda}{K}{\T}$ with $\kappa_1^{\mathrm{Hg}}\neq 0$ when the following is satisfied.
\begin{assumption}\label{assumption-1}
    Assume \eqref{Heeg}, \eqref{good}, \eqref{p-big} and \eqref{res-surj}. Furthermore, if $p$ is supersingular, assume also \eqref{split}.
\end{assumption}

We will recall the construction of such Kolyvagin systems in \Cref{kolyvagin-systems-section}.

Howard made the following conjecture in the ordinary case, and showed how $(1), (2)$ and one of the divisibilities of $(3)$ follow from the fact that $\kappa_1^{\mathrm{Hg}}\neq0$ (see \cite[Theorem B]{Howard} and \cite[Theorem 3.4.3]{Howard2}\footnote{Howard was not able to analyze the $p$ part of the characteristic ideals in the case when $N^-\neq1$ due to the factor $p^d$ in \cite[Lemma 3.4.1]{Howard2}, but as explained in \cite[Theorem 3.1]{BCK}, it is possible to take $p^d=1$ in such lemma.}). The same divisibility was extended to the supersingular case in \cite[Theorem 5.12]{Castella-Wan}.

\begin{specialtheorem}{Howard's Main Conjecture}\label{Howard-main}
Let $X=\Hom_{\Z_p}(\Cohomology{1}{\FLambda}{K}{\A},\Q_p/\Z_p).$ Then there is a torsion $\Lambda$-module $M$ such that:
\begin{enumerate}
    \item $\Char(M)=\Char(M)^\iota,$ where $\iota$ is the endormorphism induced by complex conjugation,
    \item $X\sim \Lambda\oplus M\oplus M,$
    \item $\Char(M)=\Char\left(\Cohomology{1}{\mathcal{F}_\Lambda}{K}{\T}/\Lambda\kappa_1^{\mathrm{Hg}}\right).$
\end{enumerate}
\end{specialtheorem}

\begin{remark}
We reinforce that, for the supersingular case, this is the $+$ version of \cite[Conjecture 1.2]{Castella-Wan}.
\end{remark}

We will prove the remaining divisibility of \cref{Howard-main} when $\kappa$ is primitive:
\begin{specialtheorem}{Theorem A}\label{Theorem-A}
    Suppose \cref{assumption-1}, \cref{assumption-bk}, \eqref{ap-1} and that $\kappa$ is primitive. Then we have that \cref{Howard-main} is true.
\end{specialtheorem}
This extends \cite[Theorem 1.6]{BCK} to cases where the analytic rank may be greater than $1,$ so we may also obtain new cases of \cref{Howard-main} by the work of Wei Zhang \cite[Theorem 1.1]{Wei-Zhang}, as done in \cite[Theorem 7.1]{BCK}, to obtain
\begin{theorem}\label{newresult}
    Let $E/\Q$ be an elliptic curve, $p$ a good ordinary prime number and $K$ a quadratic imaginary field with $D_K<-4,$ such that $p\nmid D_K$ and $(N_E,D_K)=1.$ Assume that \eqref{Heeg}, \eqref{p-big}, \eqref{res-surj} and \eqref{ap-1}. Assume also \cite[Hypothesis $\spadesuit$]{Wei-Zhang}, that is,
    \begin{enumerate}
        \item $E[p]$ is ramified at all primes $l\| N^+,$
        \item $E[p]$ is ramified at all primes $l\mid N^-$ with $l\equiv \pm1\mod p,$
        \item if $N$ is not square-free, then $E[p]$ is ramified at either $(a)$ at least one prime $l\mid N^-$, or $(b)$ at least in two primes $l\| N^+.$
    \end{enumerate}
    Then \cref{Howard-main} is true.
\end{theorem}
\begin{remark}\label{whatsnew}
    Note that unlike previous results on \cref{Howard-main}, we do not require $E$ to be semistable, and also allow the case of analytic rank greater than $1.$
\end{remark}

We'll also obtain a partial converse to \cref{Theorem-A}. We consider the condition
\begin{equation}[no-Tam]\tag{no split Tam}
    p\nmid c_w(E/K)\text{ for all }w\mid N^+.
\end{equation}

\begin{assumption}\label{assumption-2}
Assume \eqref{Heeg}, \eqref{good}, \eqref{split}, \eqref{p-big} and \eqref{res-surj}. Furthermore, assume that one of the following is satisfied:
\begin{enumerate}
    \item $p\nmid a_p(a_p-1)$ or
    \item $E/K$ has analytic rank $1.$
\end{enumerate}
\end{assumption}

We will then obtain
\begin{specialtheorem}{Theorem B}\label{Theorem-B}
     Assume that \cref{Howard-main} is true and that \cref{assumption-2} and \cref{assumption-bk} hold. Then $\kappa$ is primitive if and only if \eqref{no-Tam} holds.
\end{specialtheorem}

From \cref{Theorem-A}, \cref{Theorem-B} and \cite[Theorem 1.6]{BCK}, one readily obtains
\begin{specialtheorem}{Theorem C}\label{Theorem-C}
Assume \cref{assumption-2} and \cref{assumption-bk}. Then
\begin{equation}
    \kappa\text{ is primitive}\iff \text{both \cref{Howard-main} and \eqref{no-Tam} hold}.
\end{equation}
\end{specialtheorem}
\begin{proof}
\cite[Theorem 1.6]{BCK} covers the missing case of \cref{Theorem-A} where $E/K$ has analytic rank $1$ but $p\mid a_p-1,$ since in this case $p$ is automatically an ordinary prime.
\end{proof}

\subsection{Proof outline and organization of the paper}

\cref{Theorem-A} will follow from a improvement in Howard's formalism of conjugate self dual Kolyvagin Systems \cite{Howard} as mentioned in the introduction. Roughly, this improvement will consist in the following:

In \cite[Theorem 1.6.1]{Howard}, Howard establishes the inequality
\begin{equation}
    \length{M}\le \length{\Cohomology{1}{\mathcal{F}}{K}{T}\slash R\cdot \kappa_1}
\end{equation}
for a Kolyvagin system $\kappa$ over a discrete valuation ring $R.$ In the above, $T$ is a certain finitely generated $R\grpring{G_K}$-module and $M$ is such that $\Cohomology{1}{\mathcal{F}}{K}{A}\iso \mathrm{Frac}(R)\slash R\oplus M\oplus M$ for $A\defeq T\otimes_R\mathrm{Frac}(R)/R.$ This is what ultimately allows Howard to establish the corresponding divisibility in \cref{Howard-main}. We'll provide an error term for the above inequality
\begin{equation}[Howard-formula]
    \length{M}=\length{\Cohomology{1}{\mathcal{F}}{K}{T}\slash R\cdot \kappa_1}-d(\kappa),
\end{equation}
with the property that $d(\kappa)=0$ iff $\kappa$ is primitive. This mimicks the similar formula obtained in the Mazur--Rubin formalism \cite[Theorem 5.1.12(vii)]{Mazur-Rubin}, where, there, $d(\kappa)$ is replaced by $\partial^{(\infty)}(\kappa),$ the largest exponent $a$ of $p$ for which $\kappa$ is $p^a$-divisible. Although we do not know if $d(\kappa)$ must be equal to $\partial^{(\infty)}(\kappa)$ in general, they behave similarly enough that we will be able to deduce \cref{Howard-main} from the $\Lambda$-primitivity of $\kappa^{\mathrm{Hg}}$ in a similar way as in the Mazur--Rubin formalism \cite[Theorem 5.3.10(iii)]{Mazur-Rubin}. Such extension of Howard's formalism will be done in \Cref{Howard-section}.

To obtain \cref{Theorem-A} from the above, we will obtain the $\Lambda$-primitivity of $\kappa^{\mathrm{Hg}}$ from the primitivity of $\kappa.$ We note that this is the crucial place where \eqref{ap-1} will be necessary. This relation between different primitivity conditions will be established in \Cref{kolyvagin-systems-section}, where we finish the proof of \cref{Theorem-A}.

In other words, we will formalize the implications
\begin{equation}
    \kappa\text{ is primitive}\implies\kappa^{\mathrm{Hg}}\text{ is }\Lambda\text{-primitive}\implies \text{\cref{Howard-main} holds}
\end{equation}
alluded to in \cite[Remark 3.3]{BCK}.

For \cref{Theorem-B}, we note that for the case where $E/K$ has analytic rank $1,$ this is essentially already in \cite{BCK} by reversing their argument: by exploiting the close relationship between Howard Main Conjecture and the BDP Main Conjecture, we can specialize the latter in the trivial character via the anticyclotomic control theorem of \cite{JSW} together with the BDP formula \cite[Theorem 5.13]{BDP} to obtain (see \cite[Theorem 5.1]{BCK})
\begin{equation}[BCK-formula]
\mathrm{ord}_p(\abs{\Sha(E/K)[p^\infty]})=2\cdot\mathrm{ord}_p\Index{E(K)}{\Z\cdot y_K}-2\sum_{w\mid N^+}\mathrm{ord}_p(c_w(E\slash K)).
\end{equation}
But Kolyvagin's structure theorem \cite[Theorem D]{Kolyvagin} says that
\begin{equation}
\mathrm{ord}_p(\abs{\Sha(E/K)[p^\infty]})=2\cdot\mathrm{ord}_p\Index{E(K)}{\Z\cdot y_K}-2\cdot\partial^{(\infty)}(\kappa),
\end{equation}
and, thus, comparing these two formulas yield \cref{Theorem-B}, since $\partial^{(\infty)}(\kappa)=0\iff\kappa\text{ is primitive}.$

In order to allow for arbitrary analytic rank, we essentially follow the approach alluded to in \cite[Remark 1.8]{BCK}: we will prove an analog of \eqref{BCK-formula} in the case of a nontrivial anticyclotomic twist of $T$ (\cref{formula-index-2}), where the analog of Koyvagin's structure theorem is \eqref{Howard-formula}.

The analog of the special value formula \cite[Theorem 5.1]{BCK} will be proven in \Cref{control-section}, by using the anticyclotomic control theorem of Jetchev--Skinner--Wan \cite{JSW} for an anticyclotomic twist of the Galois representation of the elliptic curve.

In \Cref{main-section} we establish the relation between the Howard and BDP Main Conjectures, and use the extension of the BDP formula of \cite{Castella-Hsieh} to finish the proof of \cref{Theorem-B}.

\subsection*{Acknowledgments}
I want to thank Christopher Skinner for suggesting this investigation, for acting as my thesis advisor, and for innumerous insightful discussions throughout the preparation of this paper. I am also indebted to Francesc Castella for his interest and for multiple useful conversations about his previous works. Part of the research for this thesis was completed under the support of the 2018 Princeton Summer Research Program.

\section{Howard's Kolyvagin Systems}\label{Howard-section}
\subsection{Generalities}\label{section2.1}
We first briefly recall the setting and some notation from \cite{Howard}.

Let $K$ be a quadratic imaginary field and $p$ be a prime, and denote by $\tau\in G_\Q$ a complex conjugation. Denote by $K[n]$ the ring class field of conductor $n.$

Let $R$ be a coefficient ring (a complete, Noetherian, local ring with finite residue field of characteristic $p$) with maximal ideal $\m.$ We consider the category $\mathrm{Mod}_{R,K}$ of $R\grpring{G_K}$-modules that are unramified outside a finite set of places.

For $T\in\mathrm{Mod}_{R,K},$ we denote $\overline{T}=T/\m T,$ and consider the sets 
\begin{equation}
    \Sigma(T)\defeq\{v\text{ place of }K\colon v\mid p\infty\text{ or }T\text{ is ramified at }v\}    
\end{equation}
and $\mathcal{L}_0(T)\defeq\{\lambda\text{ inert prime of }K\}\setminus\Sigma(T).$

We recall some local conditions from \cite[Section 1.1]{Howard}. We let the strict, relaxed, finite and transverse conditions to be, respectively,
\begin{equation}
    \Cohomology{1}{*}{K_v}{T}\defeq\left\{\begin{array}{lll}0&\text{if}&*=\str\\\Cohomology{1}{}{K_v}{T}&\text{if}&*=\rel\\\Ker{\Cohomology{1}{}{K_v}{T}\to\Cohomology{1}{}{K_v^\ur}{T}}&\text{if}&*=\f\text{ and }v\not\in\Sigma(T)\\\Ker{\Cohomology{1}{}{K_v}{T}\to\Cohomology{1}{}{L[v]}{T}}&\text{if}&*=\tr\text{ and }v\in\mathcal{L}_0(T)\end{array}\right.
\end{equation}
where $L[\lambda],$ for a $\lambda\in\mathcal{L}_0(T)$ above a prime $l\in\Q,$ denotes the maximal $p$-subextension of $K[l]_\lambda/K_\lambda.$ We also denote $\Cohomology{1}{\s}{K_v}{T}\defeq\Cohomology{1}{}{K_v}{T}/\Cohomology{1}{\f}{K_v}{T}$ for $v\not\in\Sigma(T).$ We recall that when $v\in\mathcal{L}_0(T)$ we have a splitting $\Cohomology{1}{}{K_v}{T}=\Cohomology{1}{\f}{K_v}{T}\oplus \Cohomology{1}{\tr}{K_v}{T}$ (\cite[Proposition 1.1.9]{Howard}).

We consider Selmer triples $(T,\mathcal{F},\mathcal{L})$ as in \cite[Section 1.2]{Howard}. These consist of a collection of a $T\in\mathrm{Mod}_{R,K},$ local conditions $\mathcal{F}=(\mathcal{F}_v)_v$ where $\mathcal{F}_v\subseteq\Cohomology{1}{}{K_v}{T}$ are $\mathcal{F}_v=\Cohomology{1}{\f}{K_v}{T}$ when $v\not\in\Sigma(\mathcal{F})$ for some finite set $\Sigma(\mathcal{F})\supseteq\Sigma(T),$ and of a subset $\mathcal{L}\subseteq \mathcal{L}_0(T)$ that is disjoint from $\Sigma(\mathcal{F}).$ We denote by $\mathcal{N}=\mathcal{N}(\mathcal{L})$ the set of square-free products of primes in $\mathcal{L},$ with the convention that $1\in\mathcal{N}.$

We recall that given such a Selmer triple, we denote $\Cohomology{1}{\mathcal{F}}{K_v}{T}\defeq\mathcal{F}_v$ and 
\begin{equation}
    \Cohomology{1}{\mathcal{F}}{K}{T}\defeq\Ker{\Cohomology{1}{\mathcal{F}}{K}{T}\to\bigoplus_v\Cohomology{1}{}{K_v}{T}/\mathcal{F}_v}.
\end{equation}
We also recall that we propagate the Selmer condition $\mathcal{F}$ to submodules and quotients of $T$: If $0\to T'\xrightarrow{i} T\xrightarrow{q}T''\to 0,$ we denote $\Cohomology{1}{\mathcal{F}}{K_v}{T'}\defeq i^{-1}(\Cohomology{1}{\mathcal{F}}{K_v}{T})$ and $\Cohomology{1}{\mathcal{F}}{K_v}{T''}\defeq q(\Cohomology{1}{\mathcal{F}}{K_v}{T}).$

For a Selmer triple $(T,\mathcal{F},\mathcal{L})$ and a $n=abc\in\mathcal{N},$ we let the Selmer condition $\mathcal{F}_a^b(c)$ be, as usual, strict on divisors of $a,$ relaxed on divisors of $b$ and transverse on divisors of $c.$ As in \cite[Section 1.5]{Howard}, we denote
\begin{equation}
    \mathcal{H}_a^b(c)\defeq \Cohomology{1}{\mathcal{F}_a^b(c)}{K}{T}\quad\text{and}\quad\overline{\mathcal{H}_a^b}(c)\defeq \Cohomology{1}{\mathcal{F}_a^b(c)}{K}{\overline{T}}.
\end{equation}

Finally, we recall the definition of Kolyvagin systems for a Selmer triple $(T,\mathcal{F},\mathcal{L}).$ For a $l\mid\lambda\in\mathcal{L}_0(T)$, let $I_l$ denote the smallest ideal of $R$ containing $l+1$ and for which $\mathrm{Frob}_\lambda$ acts trivially on $T\slash I_lT.$ For $n\in\mathcal{N},$ we denote $I_n\defeq\sum_{l\mid n}I_l.$ These are such that for any $nl\in\mathcal{N}$ with $l$ a prime, there is a comparison isomorphism
\begin{equation}
    \phi_l^{\f\s}\colon\Cohomology{1}{\f}{K_l}{T\slash I_{nl}T}\rightiso \Cohomology{1}{\s}{K_l}{T\slash I_{nl}T}.
\end{equation}

\begin{definition}[{{\cite[Definition 1.2.3]{Howard}}}]
    Let $\KS(T,\mathcal{F},\mathcal{L})$ denote the $R$-module of Kolyvagin systems. Its elements are collections
    \begin{equation}
        \kappa=\{\kappa_n\in\Cohomology{1}{\mathcal{F}(n)}{K}{T\slash I_nT}\colon n\in \mathcal{N}\}
    \end{equation}
    such that for all $nl\in\mathcal{N}$ with $l$ a prime, we have
    \begin{equation}
        \loc{l}^\s(\kappa_{nl})=\phi_l^{\f\s}(\loc{l}(\kappa_n)),
    \end{equation}
    where $\loc{l}^\s$ denotes the composition
    \begin{equation}
        \Cohomology{1}{}{K}{T/I_{nl}T}\xrightarrow{\loc{l}}\Cohomology{1}{}{K_l}{T/I_{nl}T}\twoheadrightarrow\Cohomology{1}{\s}{K_l}{T/I_{nl}T}.
    \end{equation}
\end{definition}

We consider the following hypotheses for a pair $(T,\mathcal{F}),$ from \cite[Section 1.3]{Howard}:
\begin{enumerate}[label=(H.\arabic*),ref=H.\arabic*]
    \setcounter{enumi}{-1}
    \item \label{H.0}$T$ is a free, rank $2$ $R$-module,
    \item \label{H.1}$\overline{T}$ is an absolutely irreducible representation of $(R/\m)\grpring{G_K},$
    \item \label{H.2}there is a Galois extension $F/\Q$ such that $K\subseteq F,$ such that $G_F$ acts trivially on $T,$ and such that $\Cohomology{1}{}{F(\mu_{p^\infty})/K}{\overline{T}}=0,$
    \item \label{H.3}for every $v\in\Sigma(\mathcal{F}),$ the local condition on $\mathcal{F}$ is cartesian on $\mathrm{Quot}(T)$\footnote{This means that for any two ideals $I,J$ of $R$ and an injective map $\alpha\colon T/IT\to T/JT$ induced by multiplication with a $r\in R,$ we have $\Cohomology{1}{\mathcal{F}}{K_v}{T/IT}=\alpha^{-1}(\Cohomology{1}{\mathcal{F}}{K_v}{T/JT}).$ We note that this is automatically true for $v\not\in\Sigma(T)$ by \cite[Lemma 1.1.5]{Howard}. Although we will not directly use such \cref{H.3}, it is indirectly used in \cite[Lemma 1.3.3, Lemma 1.5.8]{Howard}},
    \item \label{H.4}there is a perfect, symmetric, $R$-bilinear pairing
    \begin{equation}
        (\ ,\ )\colon T\times T\to R(1)
    \end{equation}
    which satisfies $(s^\sigma,t^{\tau\sigma\tau^{-1}})=(s,t)^\sigma$ for every $s,t\in T$ and $\sigma\in G_K.$ We assume that the local condition $\mathcal{F}$ is its own exact orthogonal complement under the induced pairing
    \begin{equation}
        \left<\ ,\ \right>\colon\Cohomology{1}{}{K_v}{T}\times\Cohomology{1}{}{K_{\conj{v}}}{T}\to R
    \end{equation}
    for every place $v$ of $K,$
    \item \label{H.5}\begin{enumerate}[label=(\alph*)]
        \item the action of $G_K$ on $\overline{T}$ extends to an action of $G_{\Q}$ and the action of $\tau$ splits $\overline{T}=\overline{T}^+\oplus\overline{T}^-$ into one-dimensional eigenspaces,
        \item the condition $\mathcal{F}$ propagated to $\overline{T}$ is stable under the action of $G_{\Q},$
        \item the residual pairing $\overline{T}\times\overline{T}\to(R/\m)(1)$ satisfies $(s^\tau,t^\tau)=(s,t)^\tau$ for all $s,t\in T.$
    \end{enumerate}
\end{enumerate}

For this entire section, we assume that both $(T,\mathcal{F})$ and $(T^*,\mathcal{F}^*)$ satisfy the above hypotheses, where $T^*\defeq\Hom(T,R(1)),$ and $(\mathcal{F}^*)_v$ is the exact orthogonal complement of $\mathcal{F}_v$ under Tate local duality.

\subsection{Principal Artinian rings}
Throughout this subsection, we let $R$ be a principal Artinian ring of length $k.$ We assume that $\mathcal{L}\subseteq\mathcal{L}_k(T)\defeq\{v\in\mathcal{L}_0(T)\colon I_l\subseteq \m^kR\}.$ In particular, this implies that $I_nR=0$ for all $n\in\mathcal{N}.$

The goal of this section is to establish the notion of core vertices, which are the $n\in\mathcal{N}$ with $\H(n)\iso R,$ and to prove that the isomorphism class of the ideal $\kappa_nR$ does not depend on the core vertex $n$ when $p>4.$

Recall that $\H(n)$ denotes $\Cohomology{1}{\mathcal{F}(n)}{K}{T}.$ In this case, Howard constructed a generalized form of a Cassels--Tate pairing, and proved:
\begin{lemma}[{{\cite[Lemma 1.5.1]{Howard}}}]\label{H(n)-structure}
    There is an $\epsilon\in\{0,1\}$ and modules $M(n)$ such that
    \begin{equation}
        \mathcal{H}(n)\iso R^\epsilon\oplus M(n)\oplus M(n)
    \end{equation}
    for all $n\in\mathcal{N}.$
\end{lemma}

Together with global duality, the existence of this pairing also implies that:

\begin{lemma}[{{\cite[Lemma 1.5.8]{Howard}}}]\label{Lemma-1.5.8}
There are $a,b$ and $\delta$ greater than or equal to $0$ such that, in the following diagram, the cokernel of each inclusion is a direct sum of two cyclic $R$-modules of the indicated lengths.
\begin{equation}
    \begin{tikzcd}
    &\H^l(n)&\\
    \H(n)\arrow[hook]{ur}{k-a,\ k-b}&&\H(nl)\arrow[hook']{lu}[swap]{a+\delta,\ b+\delta}\\
    &\H_l(n)\arrow[hook']{lu}{a,\ b}\arrow[hook]{ru}[swap]{k-a-\delta,\ k-b-\delta}&
    \end{tikzcd}
\end{equation}
\end{lemma}

\begin{definition}[{{\cite[Definitions 1.5.2, 1.5.4]{Howard}}}]
    For $n\in\mathcal{N},$ let $\rho(n)^\pm\defeq\dim_{R/\m}\left(\overline{\mathcal{H}}(n)^\pm\right),$ and $\rho(n)\defeq\rho(n)^++\rho(n)^-.$ Finally, let $\lambda(n)\defeq \length{M(n)}$ and $\mathcal{S}(n)\defeq \m^{\lambda(n)}\mathcal{H}(n).$
\end{definition}

The above lemma readily implies the following.

\begin{corollary}\label{lambda-relation}
    For $a,b,\delta$ as in \cref{Lemma-1.5.8}, we have
    \begin{equation}
        \lambda(nl)=\lambda(n)+k-a-b-\delta.
    \end{equation}
\end{corollary}
\begin{proof}
    The two bottom inclusions in \cref{Lemma-1.5.8} give us
    \begin{equation}
    \begin{split}
        \length{\H(n)}&= \length{\H_l(n)}+a+b\quad\text{and}\\
        \length{\H(nl)}&= \length{\H_l(n)}+(k-a-\delta)+(k-b-\delta).
    \end{split}
    \end{equation}
    Since $\length{\H(n)}=k\epsilon + 2\lambda(n)$ and $\length{\H(nl)}=k\epsilon + 2\lambda(nl)$ by \cref{H(n)-structure}, subtracting the two equations above give us
    \begin{equation}
        2\lambda(n)-2\lambda(nl)=a+b-(k-a-\delta)-(k-b-\delta)=2(a+b+\delta-k)
    \end{equation}
    from which the claim follows.
\end{proof}

The next proposition is a refinement of \cite[Proposition 1.5.9]{Howard}.

\begin{proposition}\label{localization-lemma}
For $nl\in\mathcal{N}$ and $d\ge 0,$ we have
\begin{equation}
    \loc{l}(\m^d\mathcal{S}(n))=0\iff \loc{l}(\m^d\mathcal{S}(nl))=0.
\end{equation}
\end{proposition}
\begin{proof}
By \cref{Lemma-1.5.8}, the first occurs if and only if $\lambda(n)+d\ge \max(a,b)$ and the second occurs if and only if $\lambda(nl)+d\ge\max(k-a-\delta,k-b-\delta).$ Assume without loss of generality that $a\ge b.$ Now our proposition becomes
\begin{equation}
    \lambda(nl)+d\ge k-b-\delta\iff \lambda(n)+d\ge a,
\end{equation}
which is clearly true by  \cref{lambda-relation}.
\end{proof}

We assume from now on that
\begin{equation}[epsilon-1]\tag{odd rank}
    \epsilon = 1.
\end{equation}
By \cite[Proposition 1.5.5]{Howard}, this is equivalent to $\rho(n)\equiv1\mod 2$ for all $n\in\mathcal{N}.$

\begin{definition}
We call $n\in\mathcal{N}$ a core vertex if $\rho(n)=1.$ Note that, under \eqref{epsilon-1}, this is equivalent to $M(n)=0$ and also to $\H(n)\iso R.$ Denote by $\mathcal{X}$ the graph whose vertices are the core vertices, and where we have an edge between vertices $n$ and $nl$ for $nl\in\mathcal{N}$ and $l$ a prime.
\end{definition}

We will use the following lemma repeatedly in what follows.

\begin{lemma}[{{\cite[Lemma 1.5.3]{Howard}}}]\label{parity-lemma}
    For any $nl\in\mathcal{N}$ with $l$ a prime,
    \begin{enumerate}
        \item if $\loc{l}\left(\overline{\H}(n)^\pm\right)\neq 0,$ then $\rho(nl)^\pm=\rho(n)^\pm-1$ and $\loc{l}\left(\overline{\H}(nl)^\pm\right)=0,$
        \item if $\loc{l}\left(\overline{\H}(n)^\pm\right)=0,$ then $\rho(nl)^\pm=\rho(n)^\pm+1.$
    \end{enumerate}
\end{lemma}

We will accomplish our goal to prove that the ideal $\kappa_nR$ does not depend on the core vertex $n$ by proving: $(1)$ that the graph $\mathcal{X}$ is connected; $(2)$ that $\kappa_nR\iso\kappa_{nl}R$ for neighboring vertices $n,nl\in\mathcal{X}.$ We first see how we prove the latter.

\begin{proposition}\label{edges-lemma}
Let $n$ be a core vertex and $nl\in\mathcal{N}$ for a prime $l.$ Then the following are equivalent:
\begin{enumerate}
    \item $nl$ is a core vertex,
    \item $\loc{l}\colon \H(n)\rightarrow \Cohomology{1}{}{K_l}{T}$ is injective,
    \item $\loc{l}\colon\H(nl)\rightarrow \Cohomology{1}{}{K_l}{T}$ is injective,
    \item $\loc{l}\overline{\H}(n)\neq 0,$
    \item $\loc{l}\overline{\H}(nl)\neq 0.$
\end{enumerate}
Moreover, in the case where $(1)-(5)$ hold, we have a well-defined isomorphism
\begin{equation}
    (\overline{\loc{l}^\s})^{-1}\circ\overline{\phi_l^{\f\s}}\circ\overline{\loc{l}}\colon \overline{\H}(n)\rightiso\overline{\H}(nl),
\end{equation}
where we recall that $\overline{\loc{l}^\s}$ denotes the composition $\Cohomology{1}{}{K}{\overline{T}}\xrightarrow{\overline{\loc{l}}}\Cohomology{1}{}{K_l}{\overline{T}}\twoheadrightarrow\Cohomology{1}{\s}{K}{\overline{T}}.$
\end{proposition}
\begin{proof}
The kernel of maps $(2)$ and $(3)$ is $\H_l(n)$ and of maps $(4)$ and $(5)$ is $\overline{\H}_l(n).$ By \cite[Lemma 1.3.3]{Howard}, we have $\overline{\H}_l(n)=\H_l(n)[\m],$ and so one is $0$ if and only the if the other is also $0.$ Since $n$ is a core vertex, which means that $\overline{\H}(n)\iso R/\m,$ we have that $(4)$ is equivalent to its kernel being $0.$ Hence $(2)\iff(3)\iff(4)\iff(5).$

Since $n$ is a core vertex, we must have $\rho(n)^\pm=1$ and $\rho(n)^\mp=0$ for some sign $\pm.$ By \cref{parity-lemma}, this means that $nl$ is a core vertex if and only if $\loc{l}\overline{\H}(n)^\pm\neq 0.$ This last condition is equivalent to $(4)$ since $\rho(n)^\mp=0,$ which means that $\overline{\H}(n)^\pm=\overline{\H}(n).$ Hence $(1)\iff (4).$

Now assume that $nl$ is a core vertex. By \cref{parity-lemma}, $\rho(nl)^\mp=1$ and $\rho(nl)^\pm=0.$ So we have that $\loc{l}\overline{\H}(n)\subseteq \Cohomology{1}{\f}{K_l}{\overline{T}}^\pm\iso R/\m.$ But since $n$ is a core vertex, $\overline{\H}(n)\iso R/\m.$ So $(4)$ implies that in fact $\loc{l}$ induces an isomorphism
\begin{equation}
    \overline{\loc{l}}\colon \overline{\H}(n)\rightiso \Cohomology{1}{\f}{K_l}{\overline{T}}^\pm.
\end{equation}
Analogously, by $(5),$ $\loc{l}\colon \overline{\H}(nl)\to\Cohomology{1}{}{K_l}{\overline{T}}$ induces an isomorphism
\begin{equation}
    \overline{\loc{l}^\s}\colon \overline{\H}(nl)\rightiso\Cohomology{1}{\s}{K_l}{\overline{T}}^\mp.
\end{equation}
Now, since $\overline{\phi_l^{\f\s}}$ is an isomorphism that changes the eigenspaces, we have the isomorphism in the proposition.
\end{proof}

\begin{corollary}\label{same-ideals}
Let $\kappa\in\KS(T,\mathcal{F},\mathcal{L}).$ If $n$ and $m$ are connected in $\mathcal{X},$ then $\kappa_n\in\H(n)$ and $\kappa_m\in\H(m)$ generate isomorphic $R$-modules.
\end{corollary}
\begin{proof}
It suffices to consider $m=nl.$ Suppose first that $\kappa_n\neq 0.$ Let $a=\length{R\kappa_n}-1.$ Then $a\ge 0$ since $\kappa_n\neq 0,$ and we have $0\neq \pi^a\kappa_n\in\H(n)[\m].$ Then by \cref{edges-lemma}, we have
\begin{equation}
    0\neq (\phi_l^{\f\s}\circ\loc{l})(\m^a\kappa_n)=\m^a(\phi_l^{\f\s}\circ\loc{l})(\kappa_n)\subseteq \Cohomology{1}{\s}{K_l}{T}[\m]
\end{equation}
and
\begin{equation}
    \loc{l}^\s(\m^a\kappa_{nl})=\m^a\loc{l}^\s(\kappa_{nl})\subseteq \Cohomology{1}{\s}{K_l}{T}.
\end{equation}
By the Kolyvagin system relations, we have $(\phi_l^{\f\s}\circ\loc{l})(\kappa_n)=\loc{l}^\s(\kappa_{nl}).$ So these two submodules are in fact the same, and we have
\begin{equation}
    0\neq \loc{l}^\s(\m^a\kappa_{nl})\subseteq \Cohomology{1}{\s}{K_l}{T}[\m].
\end{equation}
Then, by \cref{edges-lemma}, this means that $0\neq\m^a\kappa_{nl}\subseteq\H(nl)[\m].$ This implies that $a=\length{R\kappa_{nl}}-1,$ which implies the claim.

The case that $\kappa_{nl}\neq 0$ is completely analogous, and the claim is immediate if $\kappa_n=\kappa_{nl}=0.$
\end{proof}

Now we proceed to prove that $\mathcal{X}$ is connected. For this part, it will be necessary to assume that $p>4$ in order to apply the following lemma.

\begin{lemma}\label{chebotarev-lemma}
Let $s\ge 1.$ Consider nonzero classes $c_1,\ldots,c_n\in\Cohomology{1}{}{K}{\overline{T}}$ such that $c_i$ is in the eigenspace $\epsilon_i.$ Then, if $p>n,$ there is a positive proportion of primes $l\in\mathcal{L}_s(T)$ such that $\loc{l}{c_i}$ are all nonzero.
\end{lemma}
\begin{proof}
We proceed as in \cite[Lemma 1.6.2]{Howard}. Let $F/\Q$ be the extension in \eqref{H.2} and $L$ be the Galois closure over $\Q$ of $K(T/\m^sT,\mu_{p^s}).$ Since $F$ is Galois, $L\subseteq F(\mu_{p^\infty}),$ and so the restriction map
\begin{equation}[restriction-sequence]
    \Cohomology{1}{}{K}{\overline{T}}\to\Cohomology{1}{}{L}{\overline{T}}^{\Gal{L}{K}}\iso \Hom(G_L,\overline{T})^{\Gal{L}{K}}.
\end{equation}
is an injection by \eqref{H.2}. We identify the $c_i$ with their images under restriction. Let $E_i$ be the smallest extension of $L$ with $c_i(G_{E_i})=0.$ Set $G_i=\Gal{E_i}{L}.$ Then $G_i$ is a $\F_p$ vector space with a natural action of $\Gal{L}{\Q},$ and we let $G_i^\pm$ be the eigenspaces for the action of $\tau.$ Note that $c_i(G_i^\pm)\subseteq \overline{T}^{\pm\epsilon_i}.$

We claim that the maps $c_i\colon G_i^+\to\overline{T}^{\epsilon_i}$ are nontrivial. If they were trivial, then $c_i(G_i)=c_i(G_i^-)\subseteq \overline{T}^{-\epsilon_i}.$ But then $R\cdot c_i(G_i)$ would be a $R[G_K]$-submodule of $\overline{T}$ contained in $\overline{T}^{-\epsilon_i}.$ This would imply, because of \eqref{H.1} and \eqref{H.5}(a), that $c_i(G_i)=0,$ which is not true by the injectivity of \eqref{restriction-sequence}.

Let $H_i=\{\gamma\in G_L^+\colon c_i(\gamma)=0\text{ in }\overline{T}^{\epsilon_i}\}.$ By the above, $H_i\neq G_L^+.$ Let $\mu$ be the Haar measure on $G_L^+,$ normalized with $\mu(G_L^+)=1.$ Then we have $\mu(H_i)=\frac{1}{\abs{c_i(G_L^+)}}\le\frac{1}{p}$ as $c_i(G_L^+)$ is a nonzero $\F_p$ vector space. So if $p>n,$ we have
\begin{equation}
    \mu\left(\bigcup_iH_i\right)\le \frac{n}{p}<1=\mu(G_L^+).
\end{equation}
Hence there exists a positive proportion of elements $\eta\in G_L^+$ such that $c_i(\eta)\neq 0$ for all $i.$ Write $\eta\equiv(\tau\sigma)^2\mod G_{E_1\cdots E_n}$ for some $\sigma\in\ G_L.$ We can use Chebotarev to choose a positive proportion of primes $l$ of $\Q$ such that its Frobenius class in $\Gal{E_1\cdots E_n}{\Q}$ is $\tau\sigma,$ and at which the localizations of $c_i$ are unramified. Then the images of the localizations of the $c_i$ in $\Cohomology{1}{\ur}{K_l}{\overline{T}}$ are the evaluation of $c_i$ at the Frobenius of $l,$ who in $\Gal{E_1\cdots E_n}{\Q}$ will equal $(\tau\sigma)^2=\eta,$ and hence will be nonzero.

Finally, as the Frobenius class of $l$ at $\Gal{L}{\Q}$ is $\tau,$ as in \cite[Lemma 3.5.6(i)]{Mazur-Rubin} we conclude that $l\in\mathcal{L}_s(T).$
\end{proof}

\begin{corollary}\label{chebotarev-handy}
Assume $p>3,$ and let $s\ge 1.$ Let $r(n)\defeq\max(\rho(n)^+,\rho(n)^-).$ Let $n\in\mathcal{N}$ be such that $r(n)>1$ and let $0\neq c\in\mathcal{H}(n).$ Then there is a positive proportion of primes $l\in\mathcal{L}_s(T)$ such that $\loc{l}(c)\neq 0$ and $r(nl)<r(n).$
\end{corollary}
\begin{proof}
If $\rho(n)^+,\rho(n)^-\ge 1,$ consider primes $l,$ by \cref{chebotarev-lemma}, such that
\begin{enumerate}
    \item $\loc{l}(\overline{\mathcal{H}}(n)^+)\neq 0,$
    \item $\loc{l}(\overline{\mathcal{H}}(n)^-)\neq 0,$
    \item $\loc{l}(c)\neq 0.$
\end{enumerate}
Note this is where we use $p>3.$ The third condition can be ensured by the following: consider $\tilde{c}=rc$ for some $r\in R$ in a way that $0\neq\tilde{c}\in\H(n)[\m].$ By \cite[Lemma 1.3.3]{Howard} we have $\H(n)[\m]=\overline{\H}(n)$ and so $\tilde{c}$ has nonzero projection in one of the eigenspaces, say $0\neq \tilde{c}^\pm\in\overline{\H}(n)^\pm.$ Now we can choose $l$ such that $\loc{l}(\tilde{c}^\pm)\neq 0,$ and this implies $\loc{l}(c)\neq 0.$

By \cref{parity-lemma} and $(1),(2)$ above, we have $\rho(nl)^\pm=\rho(n)^\pm-1$ for both signs $\pm,$ and hence $r(nl)=r(n)-1.$

Now consider the case where, say, $\rho(n)^\mp=0.$ Then $\rho(n)^\pm>1.$ Similarly to the above, we consider primes $l,$ by \cref{chebotarev-lemma}, such that
\begin{enumerate}
    \item $\loc{l}(\overline{\mathcal{H}}(n)^\pm)\neq 0$
    \item $\loc{l}(c)\neq 0.$
\end{enumerate}
By \cref{parity-lemma} and $(1)$ above, we have $\rho(nl)^\pm=\rho(n)^\pm-1$ and $\rho(nl)^\mp=1,$ and hence $r(nl)<r(n).$
\end{proof}

\begin{theorem}\label{connected}
If $p>4$ and $\mathcal{L}\supseteq \mathcal{L}_s(T)$ for some $s\ge 1,$ then $\mathcal{X}$ is connected.
\end{theorem}
\begin{proof}
We will prove that core vertices $n$ and $m$ are connected in $\mathcal{X}$ by induction on $\nu(\mathrm{lcm}(n,m))-\nu(\mathrm{gcd}(n,m)),$ where $\nu(t)$ denotes the number of prime divisors of $t.$  Since $n\neq m$ are square-free, we may suppose without loss of generality that there is a prime $l$ with $l\mid m$ but $l\nmid n.$

Since $n$ is a core vertex, $\rho(n)^\pm=1$ and $\rho(n)^\mp=0$ for some sign $\pm.$ By \cref{parity-lemma}, we have $\rho(nl)^\pm\in\{0,2\}$ and $\rho(nl)^\mp=1.$ If $\rho(nl)^\pm=0,$ then $nl$ is also a core vertex and we are done by induction on the pair $(n',m')=(nl,m).$

So assume $\rho(nl)^\pm=2.$ We choose a prime $r\in\mathcal{L}_s(T)$ such that
\begin{enumerate}
    \item $\loc{r}(\overline{\H}(nl)^\pm)\neq 0,$
    \item $\loc{r}(\overline{\H}(nl)^\mp)\neq 0,$
    \item $\loc{r}(\overline{\H}(n)^\pm)\neq 0,$
    \item $\loc{r}(\overline{\H}(m))\neq 0.$
\end{enumerate}
The existence of such an $r$ follows from \cref{chebotarev-lemma}, and the fact that $p>4.$

Then \cref{parity-lemma}, together with $(1)$ and $(2),$ implies that $nrl$ is a core vertex. And $(3)$ and $(4)$ with \cref{edges-lemma} imply respectively that $nr$ and $mr$ are core vertices.

This means that we have a path $n,nr,nlr$ and a path $m,mr$ in $\mathcal{X}.$ So we are done by induction on the pair $(n',m')=(nlr,mr).$
\end{proof}

As mentioned before, this concludes the proof of the following.

\begin{corollary}\label{same-ideals-artinian}
Let $\kappa\in\KS(T,\mathcal{F},\mathcal{L}),$ $p>4$ and $\mathcal{L}\supseteq \mathcal{L}_s(T)$ for some $s\ge 1.$ Then there is a $k\ge d\ge0$ such that for any core vertex $n,$ we have
\begin{equation}
    \kappa_n R=\m^d\H(n).
\end{equation}
\end{corollary}
\begin{proof}
This follows directly from \cref{connected} and \cref{same-ideals}.
\end{proof}

\begin{proposition}\label{d=k-iff-0}
Let $\kappa\in\KS(T,\mathcal{F},\mathcal{L}),$ $p>4$ and $\mathcal{L}\supseteq \mathcal{L}_s(T)$ for some $s\ge 1.$ Then $d=k$ if and only if $\kappa=0.$
\end{proposition}
\begin{proof}
It suffices to prove $\kappa=0$ when $d=k,$ that is, when $\kappa_n=0$ for all core vertices $n\in\mathcal{N}.$

For any $n\in\mathcal{N},$ we will prove that $\kappa_n$ is $0$ by induction on $r(n)\defeq \max(\rho(n)^+,\rho(n)^-).$ Indeed, if $\kappa_n\neq 0$ and $r(n)>1,$ we could choose a prime $l\in\mathcal{L},$ by \cref{chebotarev-handy}, such that $\loc{l}(\kappa_n)\neq 0$ and such that $r(nl)<r(n).$ But then the Kolyvagin system relation
\begin{equation}
    \loc{l}^\s(\kappa_{nl})=\phi^{\f\s}_l(\loc{l}(\kappa_n))
\end{equation}
would imply that $\loc{l}(\kappa_{nl})\neq0,$ which would contradict the induction hypothesis.
\end{proof}

\subsection{Discrete valuation rings}
Now let $R$ be a discrete valuation ring with $p>4.$

As in \cite[Section 1.6]{Howard}, if $\kappa$ is a Kolyvagin system for $(T,\mathcal{F},\mathcal{L}),$ for any integer $k>0$ we let $\kappa^{(k)}$ be its image as a Kolyvagin system for $(T^{(k)},\mathcal{F},\mathcal{L}^{(k)})$ where
\begin{equation}
    R^{(k)}\defeq R/\m^k,\quad T^{(k)}\defeq T/\m^kT,\quad \mathcal{L}^{(k)}\defeq \mathcal{L}\cap\mathcal{L}_k(T).
\end{equation}
We also denote by $M^{(k)}(n), \lambda^{(k)}(n),\mathcal{S}^{(k)}(n)$ the analogous objects.

We note that $\epsilon$ does not depend on $k,$ and we continue to assume \eqref{epsilon-1}. We also assume that $\mathcal{L}\supseteq\mathcal{L}_s(T)$ for some $s\ge 1.$

Now we use the results of the previous section to improve on \cite[Lemma 1.6.4]{Howard}.

\begin{lemma}\label{equal-ideals}
Let $\kappa\in\KS(T,\mathcal{F},\mathcal{L}).$ There exist $d(\kappa)\in\Z_{\ge 0}\cup\{\infty\}$ such that, if $n\in\mathcal{N}^{(2k-1)},$ then $\kappa_n^{(k)}R^{(k)}= \m^{d(\kappa)}\mathcal{S}^{(k)}(n),$ and $d(\kappa)$ is independent of $k>0.$
\end{lemma}
\begin{proof}
By \cref{same-ideals-artinian}, we have, for every $k\ge 1,$ an integer $k\ge d_k\ge 0$ such that
\begin{equation}
    \rho(n)=1\text{ and }n\in\mathcal{N}^{(k)}\implies \kappa_n^{(k)}R^{(k)}\iso \m^{d_k}R^{(k)}.
\end{equation}
Since $\mathcal{N}^{(k+1)}\subseteq \mathcal{N}^{(k)},$ it is clear that $d_k\le d_{k+1}.$ Also, if $d_k<k,$ then $d_{k+1}=d_k.$ We let $d(\kappa)=d=\lim_{k\to\infty}d_k.$ Note that we might have $d=\infty.$ Then
\begin{equation}
    \rho(n)=1\text{ and }n\in\mathcal{N}^{(k)}\implies \kappa_n^{(k)}R^{(k)}\iso \m^{d}R^{(k)}
\end{equation}
for all $k\ge 1.$

By \cite[Lemma 1.6.4]{Howard}, we have that $\kappa_n^{(k)}\in \mathcal{S}^{(k)}(n).$

We will prove by induction on $r(k,n)\defeq k+\max\{\rho(n)^+,\rho(n)^-\}$ that if $n\in\mathcal{N}^{(2k-1)},$ then $\kappa^{(k)}_nR^ {(k)}=\m^{d}\mathcal{S}^{(k)}(n).$ Note that this is immediate for the base case $\rho(n)=1,$ given the considerations above. So assume $\rho(n)>1.$ Given that $\kappa_n^{(k)}\in\mathcal{S}^{(k)}(n),$ it suffices to consider the cases where $\mathcal{S}^{(k)}(n)\neq 0.$

Using \cref{chebotarev-handy}, choose a prime $l$ such that $nl\in\mathcal{N}^{(2k-1)},$ such that $r(k,n)>r(k,nl)$ and such that
\begin{equation}[eq_nonzero-stub]
    \loc{l}(\mathcal{S}^{(k)}(n)[\m])\neq 0.
\end{equation}
For any $a\ge 0,$ we have
\begin{equation}
\begin{split}
    \m^a\kappa_n^{(k)}\neq 0&\iff\loc{l}(\m^a\kappa^{(k)}_n)\neq 0\iff\loc{l}(\m^a\kappa^{(k)}_{nl})\neq 0 \iff \loc{l}(\m^{d+a}\mathcal{S}^{(k)}(nl))\neq 0\iff\\
    &\iff \loc{l}(\m^{d+a}\mathcal{S}^{(k)}(n))\neq 0\iff \m^{d+a}\mathcal{S}^{(k)}(n)\neq 0
\end{split}
\end{equation}
where the if and only ifs follow respectively from \eqref{eq_nonzero-stub}, the Kolyvagin system relation, the induction hypothesis, \cref{localization-lemma} and \eqref{eq_nonzero-stub}.

Since $\mathcal{S}^{(k)}(n)\iso \m^{\lambda^{(k)}(n)}R^{(k)},$ this means that $\kappa_n^{(k)}R^{(k)}=\m^d\mathcal{S}^{(k)}(n),$ and the induction is complete.
\end{proof}

\begin{definition}
For $\kappa\in\KS(T,\mathcal{F},\mathcal{L}),$ we define
\begin{equation}
    \partial^{(\infty)}(\kappa)\defeq\min\{\max\{a\colon \kappa_n^{(k)}\in\m^a\H^{(k)}(n)\}\colon n\in\mathcal{N}^{(k)}\}.
\end{equation}
We call $\kappa$ \emph{primitive} if $\partial^{(\infty)}(\kappa)=0.$

Let $d(\kappa)$ be as in \cref{equal-ideals}. By the proof of the lemma, we may define it as
\begin{equation}
    d(\kappa)\defeq\min\{\max\{a\colon \kappa_n^{(k)}\in\m^a\H^{(k)}(n)\}\colon \rho(n)=1\text{ and }n\in\mathcal{N}^{(k)}\}.
\end{equation}
\end{definition}

\begin{proposition}\label{d=0iffdel=0}
Let $0\neq \kappa\in\KS(T,\mathcal{F},\mathcal{L}).$ Then $d(\kappa)\ge \partial^{(\infty)}(\kappa).$ Moreover, $d(\kappa)=0$ if and only if $\kappa$ is primitive.
\end{proposition}
\begin{proof}
By the definition of $d(\kappa)$ and $\partial^{(\infty)}(\kappa),$ it is immediate that $d(\kappa)\ge\partial^{(\infty)}(\kappa).$

For the second claim, it suffices to prove $\partial^{(\infty)}(\kappa)=0\implies d(\kappa)=0.$ We note that $\partial^{(\infty)}(\kappa)=0$ is the same as $\kappa^{(1)}\neq 0.$ Let $n\in\mathcal{N}^{(1)}$ be such that $\kappa^{(1)}_n\neq 0.$ By \cref{equal-ideals}, we have $0\neq\kappa^{(1)}_nR^{(1)}=\m^{d(\kappa)}\mathcal{S}^{(1)}(n).$ So $\m^{d(\kappa)}\mathcal{S}^{(1)}(n)\neq 0.$ This means that $d(\kappa)=0.$
\end{proof}

\begin{proposition}\label{d=infty-iff-0}
Let $\kappa\in\KS(T,\mathcal{F},\mathcal{L}).$ Then $d(\kappa)=\infty$ if and only if $\kappa=0.$
\end{proposition}
\begin{proof}
If $d(\kappa)=\infty,$ we must have $d_k=k$ for all $k.$ But then, by \cref{d=k-iff-0}, this means that $\kappa^{(k)}=0$ for all $k.$ This can only happen if $\kappa=0.$
\end{proof}
\begin{remark}
One might expect that, in fact, $d(\kappa)\isequal\partial^{(\infty)}(\kappa).$ In the case where $(T,\mathcal{F},\mathcal{L})=(T_pE,\BK,\mathcal{L}_1(T_pE))$ for $E/\Q$ an elliptic curve with $E/K$ of analytic rank $1,$ $T_pE$ its Tate module and $\BK$ the Bloch-Kato Selmer structure (see \Cref{finite-level-kol}), this equality holds true by \cref{main-dvr} below and \cite[Theorem D]{Kolyvagin}.
\end{remark}

Finally, we can improve the main result of \cite[Section 1.6]{Howard} by providing the error term $d(\kappa)$ in the formula \eqref{length-formula} below. Denote $\mathcal{D}\defeq\mathrm{Frac}(R)\slash R.$
\begin{theorem}\label{main-dvr}
Let $(T,\mathcal{F},\mathcal{L})$ be a Selmer triple such that both $(T,\mathcal{F})$ and $(T^*,\mathcal{F}^*)$ satisfy \eqref{H.0}-\eqref{H.5}. Assume $p>4$ and that there is $s\ge 1$ with $\mathcal{L}\supseteq\mathcal{L}_s(T).$ Suppose there exists $\kappa\in\KS(T,\mathcal{F},\mathcal{L})$ with $\kappa_1\neq 0.$ Then $\Cohomology{1}{\mathcal{F}}{K}{T}$ is a free rank-one $R$ module, and there is a finite $R$-module $M$ such that
\begin{equation}
    \Cohomology{1}{\mathcal{F}}{K}{A}\iso \mathcal{D}\oplus M\oplus M.
\end{equation}
Moreover, we have
\begin{equation}[length-formula]
    \length[R]{M}= \length[R]{\Cohomology{1}{\mathcal{F}}{K}{T}/R\cdot\kappa_1}-d(\kappa).
\end{equation}
In particular, we have $\length[R]{M}\le \length[R]{\Cohomology{1}{\mathcal{F}}{K}{T}/R\cdot\kappa_1}$ with equality if and only if $\kappa$ is primitive.
\end{theorem}
\begin{proof}
Except for the last two statements, this is \cite[Theorem 1.6.1]{Howard}. Note that, in particular, \eqref{epsilon-1} also follows from the assumptions of the theorem.

For $k$ sufficiently large, we have $\H^{(k)}(1)\iso R/\m^k\oplus M\oplus M,$ and by \cref{equal-ideals}, we have $\kappa_1^{(k)}R^{(k)}=\m^{d(\kappa)}\mathcal{S}^{(k)}(1).$ Let $\lambda\defeq\length{M}=\lambda^{(k)}(1).$ If $k>d(\kappa)+\lambda,$ by the injectivity of
\begin{equation}
    \H(1)/\m^k\H(1)\to\H^{(k)}(1),
\end{equation}
we have $\kappa_1 R=\m^{\lambda+d(\kappa)}\H(1).$ Hence 
\begin{equation}
    \length{M}=\length{\Cohomology{1}{\mathcal{F}}{K}{T}/R\cdot\kappa_1}-d(\kappa).
\end{equation}

The last statement now follows from \cref{d=0iffdel=0}.
\end{proof}

\subsection{Iwasawa modules}\label{section2.4}
For this section, we consider $T=T_p(E)$ to be the Tate module of an elliptic curve $E/\Q.$ We assume $(N_E,D_K)=1$ and also \eqref{good}, \eqref{p-big}, \eqref{res-surj}.

Recall that $\Lambda=\Z_p\grpring{\Gal{K_\infty}{K}}$ where $K_\infty$ is the anticyclotomic extension of $K.$ We let $\T\defeq T\otimes_{\Z_p}\Lambda,$ where the Galois action is on both factors.

Let $\P\neq p\Lambda$ be a fixed height-one prime of $\Lambda.$ Let $R=S_\P$ be the integral closure of $\Lambda/\P.$ We note that $S_\P$ is a discrete valuation ring and a finite extension of $\Z_p.$ We let $T_\P\defeq T\otimes_{\Z_p} S_\P$ with Galois action on both factors, which can also be described as $T_\P=\T\otimes_\Lambda S_\P.$

Recall that we consider a Selmer condition $\FLambda$ for $\T$ which depends on the type of reduction of $E$ over $p.$ In the ordinary case, this is \cite[Definition 3.2.2]{Howard2} and in the supersingular case this is \cite[Section 3.3]{Castella-Wan}.

Let $\mathcal{F}_\P$ be the Selmer condition on $T_\P$ induced by $\FLambda.$ It satisfies \eqref{H.0}-\eqref{H.5} by \cite[Proposition 2.1.3]{Howard} and \cite[Proof of Theorem 5.12]{Castella-Wan}.

With \cref{main-dvr}, we immediately obtain:
\begin{proposition}\label{main-heightone}
Fix a positive integer $s$ and a set of primes $\mathcal{L}\supseteq\mathcal{L}_s(T_\P),$ and suppose that there exist $\kappa\in\KS(T_\P,\mathcal{F}_\P,\mathcal{L})$ with $\kappa_1\neq 0.$ Then $\Cohomology{1}{\mathcal{F}_\P}{K}{T_\P}$ is a free, rank $1$ $S_\P$-module, and
\begin{equation}
    \Cohomology{1}{\mathcal{F}_\P}{K}{A_\P}\iso \mathcal{D}_\P\oplus M_\P\oplus M_\P
\end{equation}
where $M_\P$ is a finite $S_\P$-module with
\begin{equation}
    \length[S_\P]{M_\P}= \length[S_\P]{\Cohomology{1}{\mathcal{F}_\P}{K}{T_\P}/S_\P\cdot \kappa_1}-d(\kappa).
\end{equation}
\end{proposition}

As in \cite[Definition 5.3.9]{Mazur-Rubin}, we define the notion of a Kolyvagin system being $\Lambda$-primitive, and also of being primitive.
\begin{definition}\label{def-lambda-prim}
    Let $\kappa\in\KS(\T,\mathcal{F}_\Lambda,\mathcal{L})$ with $\mathcal{L}\supseteq\mathcal{L}_s$ for some $s.$ We say $\kappa$ is $\Lambda$\emph{-primitive} if for all height-one primes $\P$ of $\Lambda,$ there is $k=k(\P)\in\Z_{>0}$ such that the image of $\kappa$ as a Kolyvagin system for $\KS(\T/(\P,\m^k),\mathcal{F}_\Lambda,\mathcal{L}_j)$ is nonzero for all $j\ge s.$
\end{definition}

\begin{definition}
    Let $\kappa\in\KS(\T,\mathcal{F}_\Lambda,\mathcal{L})$ with $\mathcal{L}\supseteq\mathcal{L}_s$ for some $s.$ We say $\kappa$ is \emph{primitive} if its image as a Kolyvagin system for $\KS(\T/\m,\FLambda,\mathcal{L})$ is nonzero.
\end{definition}

\begin{proposition}\label{prim=>lambda-prim}
    Let $\kappa\in\KS(\T,\mathcal{F}_\Lambda,\mathcal{L})$ with $\mathcal{L}\supseteq\mathcal{L}_s$ for some $s.$ If $\kappa$ is primitive, then it is also $\Lambda$-primitive.
\end{proposition}
\begin{proof}
    We first note that the image of $\kappa$ in $\KS(\T,\mathcal{F}_\Lambda,\mathcal{L}_j)$ is also primitive for $j\ge s.$ Indeed, since $\overline{\kappa}\in\KS(\T/\m,\FLambda,\mathcal{L})$ is nonzero and $\T/\m\iso\F_p$ has length $1,$ we must have $d=0$ in \cref{same-ideals-artinian} by \cref{d=k-iff-0}. This mean that $\overline{\kappa}_n\neq 0$ for any core vertex $n.$ By \cref{chebotarev-handy} there are core vertices $n\in\mathcal{N}^{(j)}$ for any $j,$ and so by the above the image of $\kappa$ in $\KS(\T/\m,\mathcal{F}_\Lambda,\mathcal{L}_j)$ is nonzero for any $j\ge s.$
    
    Now, since all height-one primes $\P$ satisfy $(\P,\m)=\m,$ we can choose $k(\P)=1$ in \cref{def-lambda-prim} to conclude that $\kappa\in\KS(\T,\mathcal{F}_\Lambda,\mathcal{L})$ is $\Lambda$-primitive.
\end{proof}

Similarly to \cite[Theorem 5.3.10(iii)]{Mazur-Rubin}, we can now improve
\cite[Theorem 2.2.10]{Howard} to:
\begin{theorem}\label{Howard-original+}
Let $X=\Hom(\Cohomology{1}{\FLambda}{K}{\A},\Q_p/\Z_p)$ and suppose that for some $s$ there exist $\kappa\in\KS(\T,\FLambda,\mathcal{L}_s)$ with $\kappa_1\neq 0.$ Then
\begin{enumerate}[label=(\alph*)]
    \item $\Cohomology{1}{\FLambda}{K}{\T}$ is a torsion-free, rank one $\Lambda$-module,
    \item there is a torsion $\Lambda$-module $M$ such that $\Char(M)=\Char(M)^\iota$ and a pseudo-isomorphism
    \begin{equation}
        X\sim \Lambda\oplus M\oplus M,
    \end{equation}
    \item $\Char(M)\mid \Char(\Cohomology{1}{\FLambda}{K}{\T}/\Lambda\kappa_1),$ with equality if $\kappa$ is $\Lambda$-primitive. 
\end{enumerate}
\end{theorem}
\begin{proof}
The only difference to \cite[Theorem 2.2.10]{Howard} and \cite[Theorem 3.4.3]{Howard2} (in the ordinary case) and \cite[Theorem 5.12]{Castella-Wan} (in the supersingular case) is the last part of $(c).$ We recall their proofs of the divisibility in $(c)$ below, and show how one can adapt it to obtain the stronger claim when $\kappa$ is $\Lambda$-primitive.

The idea of the proof is that for almost all height-one primes $\P,$ the image $\kappa^{(\P)}$ of $\kappa$ under $\KS(\T,\FLambda,\mathcal{L}_s(\T))\to\KS(T_\P,\mathcal{F}_\P,\mathcal{L}_s(T_\P))$ will satisfy the hypothesis of \cref{main-heightone}. Even if these hypothesis do not hold for $\P$ itself, we may deform it to other primes $\mathfrak{Q}_m$ and use \cref{main-heightone} for almost all $\mathfrak{Q}_m$ to compare the exponent of $\P$ in both sides of the divisibility in $(c).$

In the proofs of both \cite[Theorem 3.4.3]{Howard2} and \cite[Theorem 5.12]{Castella-Wan}, it is shown that there is a finite set $\Sigma_\Lambda$ of height one primes of $\Lambda$ such that if $\P\not\in\Sigma_\Lambda,$ the map
\begin{equation}
    \Cohomology{1}{\FLambda}{K}{\T}/\P\Cohomology{1}{\FLambda}{K}{\T}\to\Cohomology{1}{\mathcal{F}_\P}{K}{T_\P}
\end{equation}
has finite kernel and cokernel bounded by a constant that depend only on $\Index{S_\P}{\Lambda/\P}.$ Together with $(a),$ this means that if $\P\not\in\Sigma_\Lambda,$ then $\kappa_1^{(\P)}$ is nonzero.

Let $f_\Lambda\defeq \Char(\Cohomology{1}{\FLambda}{K}{\T}/\Lambda\kappa_1).$ Let $\P$ denote a height-one prime. If $\P\neq p\Lambda,$ fix a generator $g$ of $\P$ and let $\mathfrak{Q}=(g+p^m)\Lambda.$ If $\P=p\Lambda,$ let $\mathfrak{Q}=(p+T^m)\Lambda.$ By Hensel's lemma, there exist an $N\in\Z$ such that there is an isomorphism of rings $\Lambda/\P\iso\Lambda/\mathfrak{Q}$ when $m\ge N.$ In particular, in this case, $\mathfrak{Q}$ is a height-one prime. We may also increase $N$ such that for $m\ge N$ we have that $\mathfrak{Q}$ does not divide $f_\Lambda$ nor $\Char(M),$ and that $\mathfrak{Q}\not\in\Sigma_\Lambda.$

Now let $m\ge N.$ Since $\mathfrak{Q}\not\in\Sigma_\Lambda$ and $\Lambda/\P\iso\Lambda/\mathfrak{Q}$ the kernel and cokernel of
\begin{equation}
    \Cohomology{1}{\FLambda}{K}{\T}/\mathfrak{Q}\Cohomology{1}{\FLambda}{K}{\T}\to\Cohomology{1}{\mathcal{F}_\mathfrak{Q}}{K}{T_\mathfrak{Q}}
\end{equation}
are bounded by a constant that depends only on $\P.$ Letting $d=\rank_{\Z_p}(\Lambda/\P),$ the equality $(\mathfrak{Q},\P^n)=(\mathfrak{Q},p^{mn})$ together with the above implies that, up to $O(1)$ as $m$ varies, we have
\begin{equation}
\begin{split}
    \length[\Z_p]{\Cohomology{1}{\mathcal{F}_{\mathfrak{Q}}}{K}{T_{\mathfrak{Q}}}/S_{\mathfrak{Q}}\kappa_1^{(\mathfrak{Q})}}=\length[\Z_p]{\Lambda/(f_\Lambda,\mathfrak{Q})}&=\length[\Z_p]{\Lambda/(\P^{\mathrm{ord}_\P(f_\Lambda)},\mathfrak{Q})}\\
    &=m\cdot d\cdot \mathrm{ord}_\P(f_\Lambda)
\end{split}
\end{equation}
In the same way, we have
\begin{equation}
\begin{split}
    \length[\Z_p]{M_{\mathfrak{Q}}}=\frac{1}{2}\length[\Z_p]{\Cohomology{1}{\mathcal{F}_{\mathfrak{Q}}}{K}{A_{\mathfrak{Q}}}}_{/\div}&=\frac{1}{2}\length[\Z_p]{(X/\mathfrak{Q}X)_{\Z_p\text{-}\tor}}\\
    &=m\cdot d\cdot \mathrm{ord}_\P(\Char(M))
\end{split}
\end{equation}
up to $O(1)$ as $m$ varies.

For the divisibility in $(c),$ we use the inequality in \cref{main-heightone} for $\mathfrak{Q}$ and the equations above to obtain
\begin{equation}[iwasawa-inequality]
    m\cdot d\cdot \mathrm{ord}_\P(\Char(M))\le m\cdot d\cdot \mathrm{ord}_\P(f_\Lambda)+O(1).
\end{equation}
Letting $m\to\infty$ give us $\mathrm{ord}_\P(\Char(M))\le\mathrm{ord}_\P(f_\Lambda).$

When $\kappa$ is $\Lambda$-primitive, we will be able to obtain the equality of characteristic ideals due to fact that \cref{main-heightone} provides the error term $d(\kappa^{(\mathfrak{Q})})$ in the inequality \eqref{iwasawa-inequality}, and due to the fact that such error term is well behaved under changing $m.$

The refinement of \eqref{iwasawa-inequality} is
\begin{equation}
    m\cdot d\cdot \mathrm{ord}_\P(\Char(M)) = m\cdot d\cdot \mathrm{ord}_\P(f_\Lambda) - d(\kappa^{(\mathfrak{Q})}) + O(1).
\end{equation}
To prove that $\mathrm{ord}_\P(\Char(M))=\mathrm{ord}_\P(f_\Lambda),$ it suffices to prove that $d(\kappa^{(\mathfrak{Q})})$ is bounded as $m$ varies, and for this we will use that $\kappa$ is $\Lambda$-primitive.

Exactly as in \cite[Lemma 5.3.20]{Mazur-Rubin}, we have $\kappa^{(\P)}\neq 0$ when $\kappa$ is $\Lambda$-primitive. By \cref{d=infty-iff-0}, this means that $d(\kappa^{(\P)})<\infty,$ and so we can choose some $k>0$ and $n\in\mathcal{N}^{(\P,k)}$ a core vertex such that $\kappa^{(\P,k)}_n\neq 0.$ If $m\ge k$ and $m\ge N,$ then $T_\P/\m^k\iso T_{\mathfrak{Q}}/\m^k$ as $G_K$-modules, and so we have both that $n$ is a core vertex for $\kappa^{(\mathfrak{Q})}$ and that $\kappa^{(\mathfrak{Q},k)}_n\neq 0.$ This implies that $d(\kappa^{(\mathfrak{Q})}) \le k.$
\end{proof}

\section{The Kolyvagin systems of Heegner points}\label{kolyvagin-systems-section}
In this section, we will briefly recall the construction of the several Kolyvagin systems in this paper, and will establish how they are related modulo $p.$ Then we use the results of the previous section to prove \cref{Theorem-A}.

Let $K$ be a quadratic imaginary field with $D_K<-4,$ $E/\Q$ an elliptic curve with $(N_E,D_K)=1$ satisfying \eqref{Heeg} and \eqref{res-surj}. Let $p\nmid N_ED_K$ be a prime such that \eqref{p-big} and \eqref{good}. Let $T=T_pE$ be the Tate module of $E.$

We denote by $K[m]$ the ring class field of conductor $m$ and by $K_\infty$ the anticyclotomic extension of $K,$ which is the $\Z_p$ extension inside $\cup_{k\ge0}K[p^k].$ Let $K_n/K$ be the sub-extension of $K_\infty$ of degree $p^n,$ and denote $K_\infty\cap K[1]=K_{n_0}.$ Then we let $K_n[m]\defeq K_n\cdot K[m],$ and denote by $k(n)$ the smallest $k$ such that $K_n[m]\subseteq K[mp^k].$ Note that if $n>n_0,$ this is $k(n)=n-n_0+1,$ and if $n\le n_0,$ this is $k(n)=0.$

On what follows, $\chi$ will be a finite order anticyclotomic character, that is, \begin{equation}
    \chi\colon G_K\twoheadrightarrow\Gal{K_n}{K}\to\units{\Z_p[e^{2\pi i/p^n}]}.
\end{equation}
Let $L=\Q_p[e^{2\pi i/p^n}]$ and $\O_L$ its ring of integers, with maximal ideal $\m_L.$ We assume the $n$ above is minimal, that is, that $\Ker{\chi}=G_{K_n}.$

\subsection{Finite level Kolyvagin systems}\label{finite-level-kol}
For any $m$ prime to $N,$ we let $P[m]\in E(K[m])$ be the Heegner point of conductor $m$ as in \cite[Section 5.1]{Castella-Wan}. These are points that are the images of CM points in the Jacobian $J_{N^+,N^-}$ of the Shimura curve associated to a quaternion algebra of discriminant $N^-$ and level $N^+$ under a modular parametrization
\begin{equation}
    \pi\colon J_{N^+,N^-}\to E.
\end{equation}

If a prime $l$ is split in $K,$ we denote by $\{\sigma_l,\sigma^*_l\}$ the Frobeniuses of the primes above $l$ in $K.$ These Heegner points satisfy the following norm relations for any prime $l\nmid m$ (see \cite[Proposition 5.1]{Castella-Wan}):
\begin{equation}
    \Tr_{K[ml^{k+2}]/K[ml^{k+1}]}P[ml^{k+2}]=a_lP[ml^{k+1}]-P[ml^k],\quad k\ge 0,
\end{equation}
and
\begin{equation}
    \Tr_{K[ml]/K[m]}P[ml]=\left\{\begin{array}{cl}a_lP[m]&\text{ if }l\text{ is inert,}\\(a_l-\sigma_l-\sigma^*_l)P[m]&\text{ if }l\text{ is split.}\end{array}\right.
\end{equation}
So, we may write $\Tr_{K[mp^k]/K[m]}P[mp^k]=\gamma_{k-1}P[m]$ for
\begin{equation}
    \gamma_0=\left\{\begin{array}{cc}a_p&\text{ if }p\text{ is inert,}\\a_p-\sigma_p-\sigma_p^*&\text{ if }p\text{ is split,}\end{array}\right.\quad \gamma_1=a_p\gamma_0-\delta,\quad\text{and}\quad \gamma_{k+2}=a_p\gamma_{k+1}-p\gamma_k\text{ for }k\ge1,
\end{equation}
where $\delta=\frac{\abs{\units{(\O_K/p\O_K)}}}{\abs{\units{(\Z/p\Z)}}}=\abs{\Gal{K[mp]}{K[m]}}.$ Hence
\begin{equation}[Koly-computation]
    \gamma_1=\left\{\begin{array}{cc}a_p^2-1-p&\text{ if }p\text{ is inert,}\\a_p^2-a_p(\sigma_p+\sigma_p^*)+1-p&\text{ if }p\text{ is split.}\end{array}\right.
\end{equation}

Finally, for a finite order anticyclotomic character $\chi$ as above, we define
\begin{equation}
    P[m]^\chi\defeq \sum_{\sigma\in\Gal{K[mp^{k(n)}]}{K[m]}}\sigma(P[mp^{k(n)}])\otimes\chi(\sigma^{-1})\in (E(K_n[m])\otimes_\Z\O_L(\chi^{-1})) ^{\chi\mid_{G_{K[m]}}}.
\end{equation}

From now on we consider only $m$ a square-free product of Kolyvagin primes, that is, when $m\in\mathcal{N}(\mathcal{L}_1(T)).$

Let $D_l\defeq\sum_{i=1}^li\gamma_l^i$ for a fixed generator $\gamma_l$ of $G(l)\defeq\Gal{K[l]}{K[1]}$ be the Kolyvagin derivative. This only depend on the choice of $\gamma_l$ up to an element of $\Z\cdot\Tr_l.$ We have
\begin{equation}
    (\gamma_l-1)D_l=l+1-\Tr_{K[l]/K[1]}\quad\text{in }\Z[G(l)],
\end{equation}
and we denote $D_n\defeq\prod_{l\mid n}D_l.$ Fix a set $S$ of lifts of $\Gal{K[1]}{K}$ to $G_K,$ and let $D_0=\sum_{s\in S}s$ and $D_0^\chi=\sum_{s\in S}s\otimes\chi(s^{-1}).$ Let $I_m$ be as in \Cref{section2.1}, such that $I_l=p^{\min(\ord{p}(l+1),\ord{p}a_l)}\Z.$

Now, as usual, $D_0D_mP[m]$ is fixed by $G_K$ in $E(K[m])/I_m,$ and, in the same way, $D_0^\chi D_mP[m]^\chi$ is fixed by $G_K$ in $E(K_n[m])\otimes_\Z\O_L(\chi^{-1})/I_m.$

Taking their Kummer images and using \eqref{res-surj}, we can lift them uniquely to classes 
\begin{equation}
    \kappa_m\in\Cohomology{1}{}{K}{T/I_mT}
\end{equation}
and
\begin{equation}
    \kappa_m^\chi\in\Cohomology{1}{}{K_n}{T/I_m T}^\chi\iso\Cohomology{1}{}{K}{T\otimes\chi/I_m(T\otimes\chi)},
\end{equation}
Note that, by construction, $\kappa^\chi=\kappa$ when $\chi$ is the trivial character.

\begin{proposition}
    Assume \cref{assumption-bk}. Let $\chi$ be a finite order anticyclotomic character. If $\chi$ is not trivial, we assume \cref{assumption-2}(1). Then, for $m\in\mathcal{N}(\mathcal{L}_1(T))$ we have that $\kappa_m^\chi\in\Cohomology{1}{\BK(m)}{K}{T\otimes\chi/I_m(T\otimes\chi)}.$
\end{proposition}
\begin{proof}
    Let $v$ be a prime of $K.$ We need to prove $\loc{v}(\kappa_m^\chi)\in\Cohomology{1}{\BK(m)}{K_v}{T\otimes\chi/I_m(T\otimes\chi)}.$
    
    If $v$ is inert in $K/\Q,$ then any prime above $v$ in $K[n]$ splits completely in $K_n[m]/K[m],$ which implies $I_v=I_\v$ for any $\v$ of $K_n[m]$ above $v.$ If $v\nmid m,$ the Kummer image of $D_0^\chi D_mP[m]^\chi$ in $\Cohomology{1}{}{K_n[m]}{T\otimes\chi/I_m(T\otimes\chi)}$ is unramified, and the above implies that $\kappa_m^\chi\in\Cohomology{1}{}{K}{T\otimes\chi/I_m(T\otimes\chi)}$ is also unramified.
    
    If $v=l\mid m,$ let $\lambda$ be the unique prime above $l$ in $K[l].$ We need to prove $\kappa_m^\chi$ has trivial image in $\Cohomology{1}{}{K[l]_\lambda}{T\otimes\chi/I_m(T\otimes\chi)}.$ Since $\lambda$ splits completely in $K_n[m],$ it suffices to check that $D_mP[m]^\chi$ is trivial in the semilocalization
    \begin{equation}
        \Cohomology{1}{}{(K_n[m])_l}{T\otimes\chi/I_m(T\otimes\chi)}\defeq\bigoplus_{w\mid l}\Cohomology{1}{}{(K_n[m])_w}{T\otimes\chi/I_m(T\otimes\chi)}.
    \end{equation}
    As above, we have $\loc{l}(P[m]^\chi)\in\Cohomology{1}{\ur}{(K_n[m])_l}{T\otimes\chi/I_m(T\otimes\chi)}.$ Evaluating at Frobenius, we have an isomorphism of $\Gal{K[m]}{K}$-modules
    \begin{equation}
        \Cohomology{1}{\ur}{(K_n[m])_l}{T\otimes\chi/I_m(T\otimes\chi)}\iso\bigoplus_{w\mid l}\Cohomology{1}{\ur}{(K_n[m])_l}{T\otimes\chi/I_m(T\otimes\chi)}
    \end{equation}
    where $\Gal{K_n[m]}{K}$ acts on the right hand side by permuting the summands. This means $\Gal{K[l]}{K}$ acts trivally, as primes of $K_n[m/l]$ above $l$ are totally ramified in $K_n[m].$ So $D_l$ acts by multiplication by $\frac{l(l+1)}{2},$ which is an element of $I_l.$ Hence $D_mP[m]^\chi$ has trivial image in $\Cohomology{1}{}{(K_n[m])_l}{T\otimes\chi/I_m(T\otimes\chi)}.$
    
    If $v\nmid p$ is split in $K,$ then $v$ is unramified in $K_n[m]/K,$ and this means $\loc{v}(\kappa_m^\chi)\in\Cohomology{1}{\ur}{K_v}{W\otimes\chi}.$ If $v\nmid N,$ such cohomology group is trivial. If $v\mid N^+,$ then \cref{assumption-bk} give us two cases. If $N^-\neq 1,$ then \cref{assumption-bk} and \cref{Tam-cong} guarantee such cohomology group is also trivial; if $N^-=1,$ then the proposition is proved as in \cite[Proposition 6.2]{Gross}.
    
    If $v\mid p,$ the case when $\chi$ is trivial is covered by \cite[Lemma 2.3.5]{Howard2}. If $\chi$ is non trivial, then the same proof of such lemma only proves, for $w\mid v$ a prime of $K_n,$ that $\loc{w}(\kappa_m^\chi)\in\Cohomology{1}{\f}{(K_n)_w}{T\otimes\chi/I_m(T\otimes\chi)}^{G_{K_v}}.$ But under the hypothesis of \cref{assumption-2}(1), \cref{twist-compatibility} guarantees that there is an isomorphism
    \begin{equation}
        \Cohomology{1}{\f}{(K_n)_w}{T\otimes\chi}^{G_{K_v}}\iso \Cohomology{1}{\f}{K_v}{T\otimes\chi}
    \end{equation}
    induced by restriction. This let us conclude that $\loc{v}(\kappa_m^\chi)\in\Cohomology{1}{\f}{K_v}{T\otimes\chi/I_m(T\otimes\chi)}.$
\end{proof}
\begin{remark}\label{assumption-bk-remark}
    In the proof above, \cref{assumption-bk} is only used in the case of $v\mid N^+.$ One expects that an analogue of \cite[Proposition III.3.1]{Gross-Zagier} (used in \cite[Proposition 6.2]{Gross}) should be true in the case when $N^-\neq 1,$ and such a result would remove the necessity of \cref{assumption-bk}.
\end{remark}

\begin{proposition}\label{primitivity-equivalence}
Assume \eqref{ap-1} and that $p$ is good ordinary. Then $\kappa$ is primitive if and only if $\kappa^{\chi}$ is primitive.
\end{proposition}
\begin{proof}
We have that $P[m]^\chi\equiv \Tr_{K[mp^{k(n)}]/K[m]}P[mp^{k(n)}]$ and $D_0^\chi\equiv D_0$ modulo $\m_L,$ since $\chi$ factors through a $p$-extension. If $n\le n_0,$ then $k(n)=0$ and the claim is now obvious since $P[m]^\chi=P[m],$ and hence $D_0^\chi D_mP[m]^\chi\equiv D_0D_mP[m]\mod\m_L.$ So we now assume $n>n_0,$ and note that this implies $k(n)\ge 2.$

Since $\Tr_{K[mp^{k(n)}]/K[m]}P[mp^{k(n)}]=\gamma_{k(n)-1}P[m]\equiv a_p^{k(n)-2}\gamma_1P[m]\mod p,$ we have
\begin{equation}
    D_0^\chi D_m P[m]^\chi\equiv a_p^{k(n)-2}D_0D_m\gamma_1 P[m]\mod\m_L.
\end{equation}
Since $p\nmid a_p,$ the claim now follows from the congruence below, together with \eqref{ap-1}.
\begin{equation}
    D_0D_m\gamma_1P[m]\equiv \left\{\begin{array}{cc}(a_p^2-1)D_0D_mP[m]\text{ if }p\text{ is inert,}\\(a_p-1)^2D_0D_mP[m]\text{ if }p\text{ is split}\end{array}\right.\mod p.\qedhere
\end{equation}
\end{proof}

\subsection{Iwasawa theoretic Kolyvagin systems}
We will recall the construction of the Kolyvagin systems in the following theorem. Recall that $\FLambda$ is defined, for the ordinary case, as in \cite[Definition 3.2.2]{Howard2}, and as in the $+$ condition in \cite[Section 3.3]{Castella-Wan} in the supersingular case.

\begin{theorem}\label{Howard-original}
Given \cref{assumption-1}, there exist a Kolyvagin system $\kappa^{\mathrm{Hg}}$ for $\Cohomology{1}{\FLambda}{K}{\T}$ with $\kappa_1^{\mathrm{Hg}}\neq 0.$
\end{theorem}
\begin{proof}
For the ordinary case, this is \cite[Theorem 2.3.1]{Howard}, as extended in \cite[Proposition 3.1.1]{Howard2} for the generalized Heegner Hypothesis. 
The nontriviality of the base class comes from the work of Cornut--Vatsal \cite[Theorem 1.5]{Cornut-Vatsal}. See also \cite[Theorem 3.1]{BCK}.

For the supersingular case, this is \cite[Theorem 5.13]{Castella-Wan}.
\end{proof}

We first deal with the case of good ordinary reduction. We will follow the construction of $\kappa^{\mathrm{Hg}}$ as in \cite[Theorem 2.3.1]{Howard}. Recall that $\Lambda=\Z_p\grpring{\Gal{K_\infty}{K}},$ and that $\T=T\otimes_{\Z_p}\Lambda$ with Galois action on both factors.

We define, for $n\ge 0,$
\begin{equation}
    P_n[m]\defeq \Tr_{K[mp^{k(n)}]/K_n[m]}P[mp^{k(n)}]\in E(K_n[m]).
\end{equation}

Let $H_n[m]$ be the $\Z_p[\Gal{K_n[m]}{K}]$ submodule of $E(K_n[m])\otimes\Z_p$ generated by $P[m]$ and $P_j[m]$ for $j\le n.$ Let $\mathbf{H}[m]\defeq\ilim{\Tr}H_k[m].$

By \cite[Lemma 2.3.3]{Howard}, there is an Euler system $Q[m]=\ilim{} Q_k[m]\in\mathbf{H}[m]$ for $m\in\mathcal{N}$ such that $Q_0[m]=\Phi P[m]$ where
\begin{equation}
    \Phi=\left\{\begin{array}{ll}
        (p+1)^2-a_p^2 & p\text{ is inert},\\
        (p-a_p\sigma+\sigma^2)(p-a_p\sigma_p^*+\sigma_p^{*2}) & p\text{ is split}.
    \end{array}\right.
\end{equation}

Now taking the Kolyvagin derivative and Kummer map of $Q[m],$ we can lift these classes to classes $\kappa^{\mathrm{Hg}}_m\in\Cohomology{1}{\FLambda(m)}{K}{\T/I_m\T}$ by \cite[Theorem 3.4.1]{Howard2}\footnote{We note that the factor $p^d$ there is not necessary in our case, as explained in \cite[Theorem 3.1]{BCK}.}.

\begin{proposition}\label{lambdaprimitive-ord}
Assume $p$ has good ordinary reduction, that $\kappa$ is primitive and \eqref{ap-1}. Then $\kappa^{\mathrm{Hg}}$ is primitive.
\end{proposition}
\begin{proof}
Consider the projection map $\mathbf{H}[m]\to H_0[m].$ Then $Q[m]\mapsto Q_0[m]=\Phi P[m].$ Note that
\begin{equation}
    D_0D_m\Phi P[m]\equiv \left\{\begin{array}{cc}(a_p-1)^2D_0D_mP[m]&\text{if }p\text{ is split,}\\(a_p^2-1)D_0D_mP[m]&\text{if }p\text{ is inert}
    \end{array}\right.\mod p.
\end{equation}
Since $\kappa$ is primitive, we can choose $m$ such that this is not divisible by $p,$ and this shows that $\kappa^{\mathrm{Hg}}\mod \m$ is nonzero.
\end{proof}

From now on, we consider the supersingular case. We will follow the construction of $\kappa^{\mathrm{Hg}}$ as in \cite[Theorem 5.13]{Castella-Wan}. For this, we assume \eqref{split}, that is, that $p=v\conj{v}$ in $K,$ where $v$ is the prime associated with the fixed embedding $\overline{\Q}\hookrightarrow\overline{\Q_p}.$

Let $a$ be the inertial degree of primes above $v$ on $K_{n_0}/K.$ Let $\Phi_{p^k}$ be the $p^k$-cyclotomic polynomial, and define
\begin{equation}
    \tilde{\omega}_n^\epsilon\defeq \prod_{\stackrel{1\le k\le n}{(-1)^k=\epsilon}}\Phi_{p^k}((1+X)^{p^a}),\quad \omega_n^\epsilon\defeq(1+X)^{p^a}\tilde{\omega}_n^\epsilon.
\end{equation}
Let $z_n[m]$ be the Kummer image of $P_n[m].$ Let $z_n[m]^\epsilon\in\Cohomology{1}{}{K_n[m]}{T}/(\omega_n^\epsilon)$ be such that
\begin{equation}
    \tilde{\omega}_n^{-\epsilon}z_n[m]^\epsilon=(-1)^{(n+1)/2}z_n[m]
\end{equation}
as in \cite[Section 5.1]{Castella-Wan}. Then let $Q[m]^\epsilon\defeq\displaystyle\lim_{\stackrel{\leftarrow}{n\equiv\epsilon}}z_n[m]^\epsilon\in\Cohomology{1}{}{K[m]}{\T}.$ Applying the Kolyvagin derivative for the $\epsilon=+$ classes, we obtain classes
\begin{equation}
    \kappa^{\mathrm{Hg}}_m\in\Cohomology{1}{\FLambda(m)}{K}{\T/I_m\T}.
\end{equation}

\begin{proposition}\label{lambdaprimitive-ss}
Assume $p$ has good supersingular reduction, that $\kappa$ is primitive and \eqref{split}. Then $\kappa^{\mathrm{Hg}}$ is primitive.
\end{proposition}
\begin{proof}
Consider the projection map $\Cohomology{1}{}{K[m]}{\T}\to \Cohomology{1}{}{K[m]}{T}.$ Then $Q[m]^+\mapsto z_0[m]^+.$ We have that $z_0[m]^+=z_0[m]$ is the Kummer map of $P_0[m]=\gamma_0 P[m],$ where $\gamma_0=a_p-\sigma_p-\sigma_p^*.$ Hence we have $D_0D_mP_0[m]=D_0D_m\gamma_0P[m]\equiv -2D_0D_mP[m]\mod p.$ Since $\kappa$ is primitive, we can choose $m$ such that this is not divisible by $p,$ and this shows that $\kappa^{\mathrm{Hg}}\mod \m$ is nonzero.
\end{proof}

\subsection{Theorem A}
As a direct application of the previous results, we have

\begin{theorem}\label{Howard+}
Suppose \cref{assumption-1}. Then
\begin{enumerate}
    \item $\Cohomology{1}{\FLambda}{K}{\T}$ is a rank one, torsion-free $\Lambda$-module,
    \item There is a torsion $\Lambda$-module $M$ such that $\Char(M)=\Char(M)^\iota$ and $X\sim \Lambda\oplus M\oplus M,$
    \item $\Char(M)\mid \Char(\Cohomology{1}{\FLambda}{K}{\T}/\Lambda\kappa_1^{\mathrm{Hg}})$ with equality if $\kappa^{\mathrm{Hg}}$ is $\Lambda$-primitive.
\end{enumerate}
\end{theorem}
\begin{proof}
This follows from \cref{Howard-original} together with \cref{Howard-original+}.
\end{proof}

\begin{proof}[Proof of \cref{Theorem-A}]
By \cref{Howard+}, it suffices to prove that $\kappa^{\mathrm{Hg}}$ is $\Lambda$-primitive when $\kappa$ is primitive.

If $\kappa$ is primitive, then $\kappa^{\mathrm{Hg}}$ is also primitive by \cref{lambdaprimitive-ord} and \cref{lambdaprimitive-ss}, and we conclude that $\kappa^{\mathrm{Hg}}$ is $\Lambda$-primitive by \cref{prim=>lambda-prim}.
\end{proof}

\section{Control Theorem}\label{control-section}
Let $K$ be a quadratic imaginary with $D_K<-4$ and $p=v\conj{v}$ be a prime that splits in $K,$ with $v$ determined by the fixed embedding $\overline{\Q}\hookrightarrow\overline{\Q_p}.$

The goal of this section is to apply the control theorem of \cite{JSW} to anticyclotomic twists of certain Galois representation $V$ of $G_K,$ and to interpret the terms in such formula in a way that is compatible with the BDP formula in \cite[Theorem 4.9]{Castella-Wan}.

Let $L_0/\Q_p$ be a finite unramified extension, and $V$ be a finite dimensional vector space over $L_0$ with a continuous $L_0$-linear action of $G_K.$ We denote by $\rho\colon G_K\to\mathrm{GL}_{L_0}(V)$ the action of $G_K.$ We assume that $L_0$ is unramified. Let $T$ be a $G_K$-stable $\O_{L_0}$-lattice of $V,$ and $W\defeq V/T$ its divisible quotient.

Recall that $K\subseteq K_1\subseteq K_2\subseteq\cdots$ denote the layers of the anticyclotomic extension of $K_\infty$ of $K.$ We denote by $\Psi\colon G_K\twoheadrightarrow \Gal{K_\infty}{K}$ the natural projection. Let $F_\infty$ be the image of $K_\infty$ under the fixed embedding $\overline{\Q}\hookrightarrow\overline{\Q_p},$ and $\Q_p\subseteq F_1\subseteq F_2\subseteq\cdots$ its layers.

Recall that we denote $\Lambda\defeq\Z_p\grpring{\Gal{K_\infty}{K}},$ and let $\Lambda^\vee\defeq\Hom_{\mathrm{cont}}(\Lambda,\Q_p/\Z_p).$ We denote $\T=T\otimes_{\Z_p}\Lambda$ and $\A=T\otimes_{\Z_p}\Lambda^\vee,$ with Galois action given by $\rho\otimes\Psi$ and $\rho\otimes\Psi^{-1}$ respectively.

\subsection{Anticyclotomic twists}\label{section4.1}
Let $\chi\colon G_K\xrightarrow{\Psi}\Gal{K_\infty}{K}\to\units{\Z_p[\chi]}$ be a finite order anticyclotomic character, and $n$ be such that $\Ker{\chi}=G_{K_n}.$ Note that since $L_0$ is unramified, we have that $L_0\otimes_{\Q_p}\Q_p[\chi]=L_0\cdot\Z_p[\chi],$ which we denote by $L.$ We also have $\O_{L_0}\otimes_{\Z_p}\Z_p[\chi]=\O_L.$ We denote by $V\otimes\chi$ the twist of $V,$ with coefficients in $L,$ and with corresponding $\O_L$-lattice $T\otimes\chi\defeq T\otimes_{\Z_p}\Z_p[\chi]$ and divisible quotient $W\otimes\chi\defeq W\otimes_{\Z_p}\Z_p[\chi].$

We consider the triple $(V\otimes\chi,T\otimes\chi,W\otimes\chi)$ in the context of \cite[Section 2.1]{JSW}. Note that $M$ in their notation corresponds to $M\defeq (T\otimes\chi)\otimes_{\Z_p}\Lambda^\vee$ with Galois action $\rho\otimes\chi\otimes\Psi^{-1}.$ So, in fact, we have $M=\A\otimes\chi\defeq\A\otimes_{\O_{L_0}}\O_L(\chi).$

Following \cite[Section 2.3.4]{JSW}, for $N=M$ or $N=\A$ we define a Selmer structure
\begin{equation}
    \Cohomology{1}{\Fac}{K_w}{N}\defeq\left\{\begin{array}{ll}\Cohomology{1}{}{K_{\conj{v}}}{N}&\text{if }w=\conj{v},\\\Cohomology{1}{\ur}{K_w}{N}&\text{if }w\nmid p\infty\text{ is split,}\\0&\text{else.}\end{array}\right.
\end{equation}
and also let
\begin{equation}
    X_{\ac}(M)\defeq\Hom_{\O_{L}}(\Cohomology{1}{\Fac}{K}{M},L/\O_L)\quad\text{and}\quad X_{\ac}(\A)\defeq\Hom_{\O_{L_0}}(\Cohomology{1}{\Fac}{K}{\A},L_0/\O_{L_0}).
\end{equation}

Let $\Char(X_\ac(\A))$ denote the characteristic $\Lambda\otimes_{\Z_p}\O_{L_0}$-ideal of $X_\ac(\A),$ and let $\Char(X_\ac(M))$ denote the characteristic $\Lambda\otimes_{\Z_p}\O_L$-ideal of $X_\ac(M).$ Recall that we identify $\Lambda=\Z_p\grpring{T}$ by choosing a generator $\gamma$ of $\Gal{K_\infty}{K}$ and letting $\gamma\mapsto 1+T.$

\begin{proposition}\label{comparison-M-A}
If $f_\ac$ is a generator of $\Char(X_\ac(\A)),$ then $f_\chi(T)=f_\ac((1+T)\chi^{-1}(\gamma)-1)$ is a generator of $\Char( X_{\ac}(M)).$
\end{proposition}
\begin{proof}
This follows at once from $M=\A\otimes\chi.$
\end{proof}

\subsection{Control theorem for anticyclotomic twists}
We will also assume the following about $V.$
\begin{equation}[sst]\tag{sst}
    V\text{ is semistable as a representation of }G_{K_w}\text{ for all }w\mid p,
\end{equation}
\begin{equation}[tau-dual]\tag{$\tau$-dual}
    V^\vee(1)\iso V^\tau
\end{equation}
where $V^\tau$ denotes the representation with same underlining space at $V,$ but with $G_K$-action composed with conjugation by a lift $\tau$ of complex conjugation,
\begin{equation}[geom]\tag{geom}
    V\text{ is geometric,}
\end{equation}
which means $V$ potentially semistable at all places $w\mid p$ (which is true by \eqref{sst}) and unramified away from finitely many places, and
\begin{equation}[pure]\tag{pure}
    V\text{ is pure}.
\end{equation}
This last hypothesis is a technical condition used in \cite{JSW}, and we refer to \cite[Section 2.1]{JSW} for its definition.

We also assume that
\begin{equation}[2-dim]\tag{2-dim}
    V\text{ is }2\text{ dimensional over }L_0,
\end{equation}
\begin{equation}[HT]\tag{HT}
    \text{no nonzero Hodge--Tate weight of }V\text{ is }\equiv0\mod p-1,
\end{equation}
\begin{equation}[irred]\tag{irred$_K$}
    \overline{T}\text{ is an irreducible }\O_{L_0}/\m\text{-representation of }G_K,
\end{equation}
where $\overline{T}=T/\m T$ and $\m\subseteq\O_{L_0}$ is the maximal ideal of $\O_{L_0}.$

We want to apply \cite[Theorem 3.3.1]{JSW} to $V\otimes\chi,$ and so we note that most of the hypothesis above also hold automatically for $V\otimes\chi$: since $\chi$ is a finite order anticyclotomic character, $V\otimes\chi$ also satisfy \eqref{tau-dual}, \eqref{geom}, \eqref{pure}, \eqref{2-dim} and \eqref{HT}. If $\m_L$ denotes the maximal ideal of $\O_L,$ we have that $\chi\equiv1\mod\m_L,$ since $\chi$ factors through a $p$-extension, and hence $\overline{T\otimes\chi}\iso \overline{T}\otimes\O_L/\m_L$ as $G_K$ representations. Hence $T\otimes\chi$ also satisfy \eqref{irred}.

The condition \eqref{sst} may not hold anymore for $V\otimes\chi.$ However, such condition is only used in \cite[Proposition 3.3.7 case 3(b)]{JSW}. As we will explain below, it is going to enough for us that \eqref{sst} holds for $V.$

Finally, we assume the two following hypotheses that do depend on the choice of $\chi.$
\begin{equation}[corank-1]\tag{corank $1$}
    \Cohomology{1}{\BK}{K}{W\otimes\chi}_\div\iso L/\O_L\quad\text{and}\quad\Cohomology{1}{\f}{K_w}{W\otimes\chi}\iso L/\O_L\quad\text{for }w\mid p,
\end{equation}
and
\begin{equation}[sur]\tag{sur}
    \Cohomology{1}{\BK}{K}{W\otimes\chi}_\div\xtwoheadrightarrow{\loc{w}}\Cohomology{1}{\f}{K_w}{W\otimes\chi}\quad\text{for }w\mid p.
\end{equation}

\begin{theorem}\label{anticyclotomic-control}
    If $f_\ac(T)$ is a generator of $\Char(X_{\ac}(\A)),$ then
    \begin{equation}
        \abs{\O_L/f_\ac(\chi^{-1}(\gamma)-1)}=\abs{\Sha_{\mathrm{BK}}(W\otimes\chi/K)}\cdot(\abs{\delta_v(\chi)})^2\cdot C(W\otimes\chi)
    \end{equation}
    where we have that $\delta_v(\chi)\defeq\Coker{\Cohomology{1}{\BK}{K}{T\otimes\chi}\to\Cohomology{1}{\f}{K_v}{T\otimes\chi}/\Cohomology{1}{}{K_v}{T\otimes\chi}_\tor},$ that $\Sha_{\mathrm{BK}}(W\otimes\chi/K)\defeq \Cohomology{1}{\BK}{K}{W\otimes\chi}/ \Cohomology{1}{\BK}{K}{W\otimes\chi}_\div,$ and that
    \begin{equation}
        C(W\otimes\chi)\defeq \abs{\Cohomology{0}{}{K_v}{W\otimes\chi}}\cdot\abs{\Cohomology{0}{}{K_{\conj{v}}}{W\otimes\chi}}\cdot \prod_{w\in S_p}\abs{\Cohomology{1}{\ur}{K_w}{W\otimes\chi}}
    \end{equation}
    where $S_p$ is the finite set of places where $V\otimes\chi$ is ramified, excluding the places $w\mid p.$
\end{theorem}
\begin{proof}
    This follows from \cite[Theorem 3.3.1]{JSW}, \cref{comparison-M-A} and \cite[Proposition 3.2.1]{JSW}.
    
    As mentioned before, we must justify why \cite[Theorem 3.3.1]{JSW} holds since we do not have \eqref{sst} for $V\otimes\chi$ in general.
    
    The only change is in case 3(b). We will show that it suffices to know that $V$ is semistable for the argument to go through. Let $P_v=\Ker\Psi|_{G_{K_v}}$ as in their proof. The argument in \cite{JSW} relies on \eqref{sst} only to prove that $(V\otimes\chi)^{P_v}=0.$

    Since $\chi$ factors through $K_{\infty}/K,$ it is trivial on $P_v.$ So $(V\otimes\chi)^{P_v}=V^{P_v}\otimes_{L_0}L(\chi),$ and the argument in \cite[Theorem 3.3.1]{JSW} shows that $V^{P_v}=0$ since $V$ satisfies \eqref{sst}.
\end{proof}

\subsection{Special value formula}
The goal of this section is to compute the term $\abs{\delta_v(\chi)}$ in \cref{anticyclotomic-control} in terms of both the global index and the Bloch--Kato logarithm of some $c\in\Cohomology{1}{\f}{K}{T\otimes\chi}$ with non-torsion localization. When we later apply these results to the case of the Galois representation of an elliptic curve, we will take $c=\kappa_1^\chi$ as in \Cref{finite-level-kol}.

If $V'$ is a finite dimensional $G_{\Q_p}$ representation, we let  $\dperiod{*}{F_k}{V'}\defeq(V'\otimes_{\Q_p}\mathbb{B}_*)^{G_{F_k}}$ for $*\in\{\mathrm{cris},\mathrm{dR}\},$ where $\mathbb{B}_{\mathrm{cris}},\mathbb{B}_{\mathrm{dR}}$ are Fontaine's rings of cristalline and deRham $p$-adic periods, respectively. If $V'$ has coefficients in $L',$ then $\dperiod{*}{F_k}{V'}$ is a free $L'\otimes_{\Q_p}F_k$-module. The filtration on $\mathbb{B}_{\mathrm{dR}}$ induces a filtration in $\ddr{F_k}{V'},$ and the action of the cristalline Frobenius $\Phi$ in $\mathbb{B}_{\mathrm{cris}}$ induces an action of $\Phi$ in $\dcris{F_k}{V'}.$

In addition to the conditions on $V$ imposed so far, we will also assume that
\begin{equation}[cris]\tag{cris}
    V\text{ is crystalline as a representation of }G_{K_w},\text{ for }w\mid p,
\end{equation}
that
\begin{equation}[rank-1]\tag{rank 1}
    \Cohomology{1}{\f}{K}{T\otimes\chi}\iso\O_L\quad\text{and}\quad\Cohomology{1}{\f}{K_w}{V}\iso L_0\text{ for }w\mid p.
\end{equation}
and that
\begin{equation}[no-inv]\tag{no inv}
    V^{G_{F_k}}=(V^\vee(1))^{G_{F_k}}=0\quad\text{for all }k\ge0.
\end{equation}

We also assume that
\begin{equation}[e=f]\tag{Euler factor}
    \dcris{F_k}{V}^{\Phi=1}=0\quad\text{for all }k\ge 0
\end{equation}
and
\begin{equation}[e=f-chi]\tag{$\chi$-Euler factor}
    \dcris{}{V\otimes\chi}^{\Phi=1}=0.
\end{equation}
\begin{remark}\label{redundancy}
    We note how some of the conditions imposed so far are redundant.
    \begin{enumerate}
        \item The second part of \eqref{corank-1} follows from \eqref{tau-dual}, \eqref{no-inv} and the first part of \eqref{rank-1}: the Tate dual of $\Cohomology{1}{\f}{K_w}{W\otimes\chi}$ under Tate local duality is $\Cohomology{1}{/\f}{K_w}{T^\vee(1)\otimes\chi^{-1}}=\Cohomology{1}{/\f}{K_{\conj{w}}}{T\otimes\chi}.$ We have $\dim_L\Cohomology{1}{}{K_{\conj{w}}}{V}=2$ by \eqref{no-inv} with local Tate duality and the local Euler characteristic formula. Together with the first part of \eqref{rank-1}, this means that $\Cohomology{1}{/\f}{K_{\conj{w}}}{T\otimes\chi}$ has rank $1$ over $\O_L.$ Since it is also torsion-free, we have $\Cohomology{1}{/\f}{K_{\conj{w}}}{T\otimes\chi}\iso\O_L$ and hence $\Cohomology{1}{\f}{K_w}{W\otimes\chi}\iso L/\O_L.$
        \item \eqref{e=f} and \eqref{e=f-chi} follow from \eqref{tau-dual}, \eqref{pure} and \eqref{cris}: The conditions \eqref{pure} and \eqref{cris} imply that the eigenvalues of $\Phi$ in $\dcris{F_k}{V}$ are Weil numbers of absolute value $(\abs{\O_{F_k}/\m_k})^{m/2}$ for some $m\in\Z,$ where $\m_k$ is the maximal ideal of $\O_{F_k}.$ Since this is also true for the completion of $K_\infty$ at any prime above $p$ with the same $m,$ \eqref{tau-dual} implies that $m/2=-m/2-1,$ and so $m=-1.$ In particular, $\dcris{F_k}{V}^{\Phi=1}=0,$ and so \eqref{e=f} holds. The same argument works for $V\otimes\chi$ since it is also \eqref{tau-dual}, \eqref{pure} and \eqref{cris}.
    \end{enumerate}
\end{remark}

The conditions \eqref{e=f} and \eqref{e=f-chi} imply that we have isomorphisms
\begin{equation}
    \frac{\dr{F_k}{V}}{\Fil{0}{\dr{F_k}{V}}}\xrightiso{\exp}\Cohomology{1}{\f}{F_k}{V}\quad\text{and}\quad \frac{\dr{}{V\otimes\chi}}{\Fil{0}{\dr{}{V\otimes\chi}}}\xrightiso{\exp}\Cohomology{1}{\f}{\Q_p}{V\otimes\chi},
\end{equation}
by \cite[Corollary 3.8.4]{Bloch-Kato} and \cite[Definition 3.10]{Bloch-Kato}.

So under the conditions \eqref{e=f} and \eqref{2-dim}, we have that the second part of \eqref{rank-1} is implied by $V$ having Hodge--Tate weights $(\lambda_1,\lambda_2)$ with $\lambda_1\le 0<1\le\lambda_2$ as a representation of $G_{K_v}$: this implies the second part of \eqref{rank-1} for $w=v$ by the exponential map above, and by \eqref{tau-dual} the Hodge--Tate weights of $V$ as a representation of $G_{K_{\conj{v}}}$ are $(1-\lambda_2,1-\lambda_1),$ which also satisfy $(1-\lambda_2)\le 0<1\le(1-\lambda_1),$ and hence the second part of \eqref{rank-1} is also satisfied for $w=\conj{v}.$

We also assume
\begin{equation}[ordinary]\tag{ordinary}
    V\text{ is ordinary at }v.
\end{equation}
By this we mean that there is a $G_{K_v}$-stable subspace $V^+$ whose Hodge--Tate weight is $\lambda_2.$ We denote $V^-\defeq V/V^+$ and $T^\pm\defeq T\cap V^\pm,$ $W^\pm\defeq V^\pm/T^\pm.$

Finally, we assume
\begin{equation}[not-anomalous]\tag{not-anomalous}
    \Cohomology{0}{}{K_v}{W^-}=0.
\end{equation}

Let $\omega\in\Fil{1}{\ddr{}{V^\vee}}$ be a nonzero period. When we later apply these results to the case of the Galois representation of an elliptic curve, $\omega$ will be a N\'eron differential. Given such a $\omega,$ we consider the isomorphism
\begin{equation}
    \log_\omega\colon\Cohomology{1}{\f}{\Q_p}{V\otimes\chi}\xrightiso{\log}\frac{\ddr{}{V\otimes\chi}}{\Fil{0}{\ddr{}{V\otimes\chi}}}\xrightiso{\pair{\cdot}{\omega}}\ddr{}{L(\chi)}=(F_s\otimes_{\Q_p}L(\chi))^{G_{\Q_p}}.
\end{equation}
We note that the second map is an isomorphism since it is a nonzero map of vector spaces of same dimension by \eqref{rank-1}. Here, $s$ is such that $\Ker{\chi\mid_{G_{K_v}}}=G_{F_s}.$

We also consider
\begin{equation}
    \log_\omega\colon\Cohomology{1}{\f}{F_k}{V}\xrightiso{\log}\frac{\ddr{F_k}{V}}{\Fil{0}{\ddr{F_k}{V}}}\xrightiso{\pair{\cdot}{\omega}}\ddr{F_k}{L_0}=F_k\otimes_{\Q_p}L_0,
\end{equation}
and we denote
\begin{equation}
    A(T,\omega,k)\defeq\frac{\Index{\O_{F_k}\otimes_{\Z_p}\O_{L_0}}{\log_\omega(\Cohomology{1}{\f}{F_k}{T})}}{\abs{\Cohomology{1}{\f}{F_k}{T}_\tor}}.
\end{equation}

\begin{theorem}\label{first-special-value-formula}
Let $c\in\Cohomology{1}{\BK}{K}{T\otimes\chi}$ be any class with non-torsion localization at $v.$ Then we have
\begin{equation}
    \abs{\delta_v(\chi)}\cdot\abs{\Cohomology{0}{}{K_v}{W\otimes\chi}}=\frac{\abs{\O_{F_s}\otimes_{\Z_p[G_s]}\O_L(\chi^{-1})/\log_\omega(\loc{v}z)}}{\abs{\O_L/p^s}\cdot\Index{\Cohomology{1}{\BK}{K}{T\otimes\chi}}{\O_L\cdot z}}\cdot\frac{A(T,\omega,s-1)}{A(T,\omega,s)}.
\end{equation}
\end{theorem}
\begin{proof}
Note that \eqref{no-inv} implies that $\abs{\Cohomology{0}{}{K_v}{W\otimes\chi}}=\abs{\Cohomology{1}{\f}{K_v}{T\otimes\chi}}_\tor.$

As $\Cohomology{1}{\BK}{K}{T\otimes\chi}$ is a free rank one $\O_L$-module by \eqref{rank-1}, the existence the above class $c\in\Cohomology{1}{\BK}{K}{T\otimes\chi}$ with non-torsion localization means that we have an injection
\begin{equation}
    \Cohomology{1}{\BK}{K}{T\otimes\chi}\hookrightarrow \Cohomology{1}{\f}{K_v}{T\otimes\chi}_{/\tor}.
\end{equation}
So we may write
\begin{equation}
    \abs{\delta_v(\chi)}=\frac{\Index{\Cohomology{1}{\f}{K_v}{T\otimes\chi}_{/\tor}}{\O_L\cdot \loc{v}c}}{\Index{\Cohomology{1}{\BK}{K}{T\otimes\chi}}{\O_L\cdot c}}.
\end{equation}

So it remains to show that
\begin{equation}[eq_computation-twist]
    \Index{\Cohomology{1}{\f}{K_v}{T\otimes\chi}}{ \O_L\cdot\loc{v}c}\isequal\frac{\abs{\O_{F_s}\otimes_{\Z_p[G_s]}\O_L(\chi^{-1})/\log_\omega(\loc{v}c)}}{\abs{\O_L/p^s}}\cdot\frac{A(T,\omega,s-1)}{A(T,\omega,s)}.
\end{equation}
This calculation is done in \Cref{appendix-section}.
\end{proof}

\begin{definition}
For a character $\chi\colon\Gal{F_n}{F_m}\to\units{\C_p},$ we define its (Galois) Gauss sum $\mathfrak{g}(\chi),$ well defined up to a unit in $\Z_p[\chi],$ to be the Gauss sum of the Dirichlet character 
\begin{equation}
    \units{(\Z/p^{n-m+1}\Z)}\rightiso \Z/(p-1)p^{n-m}\Z\twoheadrightarrow \Z/p^{n-m}\Z\rightiso\Gal{F_n}{F_m}\xrightarrow{\chi}\units{\C_p}
\end{equation}
for any choice of the two isomorphisms. In other words, for a choice of generator $g\in\units{(\Z/p^{n-m+1}\Z)}$ and of generator $\xi\in\Gal{F_n}{F_m},$ we let
\begin{equation}
    \mathfrak{g}(\chi)=\sum_{j=0}^{(p-1)p^{n-m}}\chi(\xi^{j})e^{2\pi i g^j/p^{n-m+1}}.
\end{equation}
\end{definition}

\begin{corollary}\label{formula-index}
    Let $c\in\Cohomology{1}{\mathcal{F}}{K}{T\otimes\chi}$ be any class with non-torsion localization at $v.$ Let $p^s$ be the (Galois) conductor of $\chi\lvert_{G_{K_v}}.$ Then we have
    \begin{equation}
    \begin{split}
        &f_\ac(\chi^{-1}(\gamma)-1)\equal\left(\frac{\log_{\omega,v}z}{\mathfrak{g}\left(\chi\lvert_{\Gal{F_s}{F_{f_0}}}\right)}\right)^2
        \iff\\
        &\hspace{50pt}\frac{\Index{\Cohomology{1}{\BK}{K}{T\otimes\chi}}{\O_L\cdot z}^2}{\prod_{w\mid N^+}c_w(W\otimes\chi)}= \abs{\Sha_{\mathrm{BK}}(W\otimes\chi/K)}\cdot\frac{A(T,\omega,s-1)}{A(T,\omega,s)}\cdot\frac{\abs{\Cohomology{0}{}{K_{\conj{v}}}{W\otimes\chi}}}{\abs{\Cohomology{0}{}{K_v}{W\otimes\chi}}},
    \end{split}
    \end{equation}
    where
    \begin{equation}
        \log_{\omega,v}\colon\Cohomology{1}{\f}{K}{T\otimes\chi}\xrightarrow{\loc{v}}\Cohomology{1}{\f}{K_v}{T\otimes\chi}\xrightarrow{\log_\omega}(F_s\otimes_{\Q_p}L(\chi))^{G_{\Q_p}}\to F_s\cdot L
    \end{equation}
    and
    \begin{equation}
        c_w(W\otimes\chi)\defeq\abs{\Cohomology{1}{\ur}{K_w}{W\otimes\chi}}
    \end{equation}
\end{corollary}
\begin{proof}
From \cref{anticyclotomic-control} and \cref{first-special-value-formula}, it suffices to prove that
\begin{equation}
    \frac{\abs{\O_{F_s}\otimes_{\Z_p[G_s]}\O_L(\chi^{-1})/\log_\omega(\loc{v}z)}}{\abs{\O_L/p^s}}=\abs{\O_L/\frac{u\cdot\log_{\omega,v}z}{\mathfrak{g}\left(\chi\lvert_{\Gal{F_s}{F_{f_0}}}\right)}}.
\end{equation}
where $u\in\units{(\Z_p^\ur)}$ is such that $u\cdot \log_{\omega,v} z\in\O_L.$ This is done in \Cref{app-sec2}.
\end{proof}

We also note that the vanishing of $c_w(W\otimes\chi)$ is equivalent to that of $c_w(W).$

\begin{proposition}\label{Tam-cong}
    For any $\lambda\in S_p$ we have
    \begin{equation}
        c_\lambda(W\otimes\chi)=1\iff c_\lambda(W)=1.
    \end{equation}
\end{proposition}
\begin{proof}
We need to show that
\begin{equation}
    \Cohomology{1}{\ur}{K_\lambda}{W\otimes\chi}=0\iff \Cohomology{1}{\ur}{K_\lambda}{W}=0.
\end{equation}
Since they are finite modules, it suffices to show that
\begin{equation}
    \Cohomology{1}{\ur}{K_\lambda}{W\otimes\chi}/\m_L=0\iff\Cohomology{1}{\ur}{K_\lambda}{W}/\m_{L_0}=0.
\end{equation}

By definition, the unramified cohomology is $\Cohomology{1}{\ur}{K_\lambda}{W}=\Cohomology{1}{}{K_\lambda^{\ur}/K_\lambda}{W^{I_\lambda}}.$ As $\chi$ is unramified at $\lambda,$ $\chi$ acts trivially on $I_\lambda,$ so $\Cohomology{1}{\ur}{K_\lambda}{W\otimes\chi}=\Cohomology{1}{}{K_\lambda^\ur/K_\lambda}{W^{I_\lambda}\otimes_{\O_{L_0}}\O_L(\chi)}.$ Since $K_\lambda^{\ur}/K_\lambda$ is procyclic, we have $\abs{\Cohomology{1}{\ur}{K_\lambda}{W}}=W^{I_\lambda}/(\mathrm{Frob}_\lambda-1).$

Now $\Cohomology{1}{\ur}{K_\lambda}{W\otimes\chi}/\m_L= (W^{I_\lambda}\otimes_{\O_{L_0}}\O_L(\chi))/(\m_L,\mathrm{Frob}_\lambda-1).$ Since $\chi\equiv1\mod\m_L,$ this is \begin{equation}
    (W^{I_\lambda}\otimes_{\O_{L_0}}\O_L)/(\m_L,\mathrm{Frob}_\lambda-1)=(W^{I_\lambda}/(\mathrm{Frob}_\lambda-1))\otimes_{\O_{L_0}}\O_L/\m_L=\Cohomology{1}{\ur}{K_\lambda}{W}\otimes_{\O_{L_0}}\O_L/\m_L.
\end{equation}
Hence $\Cohomology{1}{\ur}{K_\lambda}{W\otimes\chi}/\m_L=\Cohomology{1}{\ur}{K_\lambda}{W}\otimes_{\O_{L_0}}\O_L/\m_L$ and the claim follows.
\end{proof}

\subsection{The case of elliptic curves}
Now we restrict to the case where $T=T_pE$ is the Tate module of certain elliptic curves $E/\Q.$ We will assume that $E$ satisfy \cref{assumption-2}$(1).$

We note that \eqref{cris}, and hence \eqref{sst}, follow from $p\nmid N_E.$ \eqref{tau-dual} holds since $V^\tau\iso V,$ as $V$ is a representation of $G_\Q,$ and $V\iso V^\vee(1).$ The hypotheses \eqref{irred} and \eqref{no-inv} follow from \eqref{res-surj}, and \eqref{geom}, \eqref{pure}, \eqref{2-dim} and \eqref{HT} are clearly satisfied. The conditions \eqref{e=f} and \eqref{e=f-chi} are true by \cref{redundancy}$(2),$ and \eqref{ordinary} and \eqref{not-anomalous} follow from $p\nmid a_p(a_p-1).$ Also, note that in this case $L_0=\Q_p$ is unramified.

Finally, we will choose a finite order anticyclotomic character $\chi$ such that \eqref{corank-1}, \eqref{sur} and \eqref{rank-1} hold. We will show these conditions follow from
\begin{equation}[chi-nontrivial]\tag{$\chi$-nontrivial}
    \loc{v}(\kappa_1^{\chi})\text{ and }\loc{v}(\kappa_1^{\chi^{-1}})\text{ are non-torsion},
\end{equation}
where $\kappa_1^\chi$ was constructed in \Cref{finite-level-kol}.

\begin{proposition}
    If \eqref{chi-nontrivial} holds, then so do \eqref{corank-1}, \eqref{sur} and \eqref{rank-1}.
\end{proposition}
\begin{proof}
    From \eqref{chi-nontrivial}, we have also that $\kappa_1^\chi\neq 0.$ By \cref{main-dvr}, this implies both \eqref{rank-1}, as the Hodge--Tate weights are $(\lambda_1,\lambda_2)=(0,1)$, and the first part of \eqref{corank-1}. By \cref{redundancy}$(1),$ we conclude that \eqref{corank-1} holds.
    
    By the above, we have that $\Cohomology{1}{\f}{K_w}{W\otimes\chi}\iso L/\O_L$ for $w\mid p.$ So to conclude \eqref{sur}, it suffices to prove that the map $\Cohomology{1}{\BK}{K}{W\otimes\chi}_{\div}\to\Cohomology{1}{\f}{K_w}{W\otimes\chi}$ is nonzero for $w\mid p.$ As $\Cohomology{1}{\f}{K_{\conj{v}}}{W\otimes\chi}=\Cohomology{1}{\f}{K_v}{W\otimes\chi^{-1}},$ it suffices to prove this for $w=v,$ as \eqref{chi-nontrivial} is symmetric in $\chi$ and $\chi^{-1}.$
    
    Let $n$ be such that $\Ker{\chi}=G_{K_n},$ and denote $K'=K_n.$ and let $\v$ denote the place above $v$ in $K'$ given by the embedding $\overline{\Q}\hookrightarrow\overline{\Q_p}.$ Consider the diagram
    \begin{equation}
        \begin{tikzcd}
        \left(E(K')\otimes L(\chi)/\O_L\right)^{G_K}\arrow[hook]{r}\arrow{d}&(\Cohomology{1}{\BK}{K'}{W}\otimes\O_L(\chi))^{G_K}_\div\arrow{d}{\loc{\v}\otimes\mathrm{id}}\\
        \left(E(K'_\v)\otimes L(\chi)/\O_L\right)^{G_{K_w}}\arrow[hook]{r}&\left(\Cohomology{1}{\f}{K'_\v}{W}\otimes\O_L(\chi)\right)^{G_{K_w}}
        \end{tikzcd}
    \end{equation}
    where the horizontal maps are the Kummer maps.
    
    We recall from \Cref{finite-level-kol} that $\kappa_1^{\chi^{-1}}$ is the image of a point
    \begin{equation}
        P\defeq D_0^{\chi^{-1}} P[1]^{\chi^{-1}}\in\left( E(K')\otimes\O_L(\chi)\right)^{G_K}
    \end{equation}
    under the Kummer map. Since $\loc{v}(\kappa_1^{\chi^{-1}})$ is non-torsion, this means we can choose $m>>0$ such that $P\otimes p^{-m}\in\left(E(K')\otimes L(\chi)/\O_L\right)^{G_K}$ has nonzero image in the vertical map. This shows that the map $\loc{\v}\otimes\mathrm{id}$ is nontrivial.
    
    We claim that the natural restriction map induces a map
    \begin{equation}[restriction]
        \Cohomology{1}{\BK}{K}{W\otimes\chi}_\div\to(\Cohomology{1}{\BK}{K'}{W}\otimes\O_L(\chi))^{G_K}_\div.
    \end{equation}
    For $l\nmid p$ we have $\Cohomology{1}{\f}{K_l}{W\otimes\chi}=0,$ and so it suffices to check the local conditions above $p.$ Since $(V\otimes\chi)\otimes_{\Q_p}B_{\mathrm{cris}}=(V\otimes_{\Q_p} B_{\mathrm{cris}})\otimes
    L(\chi),$ the compatibility at the places above $p$ follow from the commutativity of the diagram below for any $\w$ in $K'$ above $w\mid p.$
    \begin{equation}
       \begin{tikzcd} \Cohomology{1}{}{K_w}{V\otimes\chi}\arrow{r}\arrow{d}&\Cohomology{1}{}{K_\w}{V}\otimes L(\chi)\arrow{d}\\
       \Cohomology{1}{}{K_w}{(V\otimes\chi)\otimes_{\Q_p}B_{\mathrm{cris}}}\arrow{r}&\Cohomology{1}{}{K_\w}{V\otimes_{\Q_p}B_{\mathrm{cris}}}\otimes L(\chi)
       \end{tikzcd}
    \end{equation}
    
    By inflation-restriction and \eqref{no-inv}, \eqref{restriction} has finite cokernel, and hence is surjective. So we have a commutative diagram
    \begin{equation}
        \begin{tikzcd}
        \Cohomology{1}{\BK}{K}{W\otimes\chi}_\div\arrow[two heads]{r}\arrow{d}{\loc{v}}&(\Cohomology{1}{\BK}{K'}{W}\otimes\O_L(\chi))^{G_K}_\div\arrow{d}{\loc{\v}\otimes\mathrm{id}}\\
        \Cohomology{1}{\f}{K_v}{W\otimes\chi}\arrow{r}&\left(\Cohomology{1}{\f}{K'_\v}{W}\otimes\O_L(\chi)\right)^{G_{K_w}}
        \end{tikzcd}
    \end{equation}
    Since $\loc{\v}\otimes\mathrm{id}$ is nonzero, we conclude that $\loc{v}$ is also nonzero.
\end{proof}

Now we compute the term $\frac{A(T,\omega,s-1)}{A(T,\omega,s)}$ for $\omega=\omega_E$ a N\'eron differential. We have a commutative diagram (\cite[Example 3.11]{Bloch-Kato})
\begin{equation}[Bloch-Kato-Kummer]
    \begin{tikzcd}
    \Cohomology{1}{\f}{F_k}{V}&\Q\otimes \hat{E}(F_k)\arrow{l}{\sim}[swap]{\kappa}\\
    \frac{\dr{F_k}{V}}{\Fil{0}{\dr{F_k}{V}}}\arrow{u}{\mathrm{exp}}[swap]{\sim}&\arrow{l}{\sim}\arrow{u}{\exp}[swap]{\sim}\mathrm{Tan}(E(F_k))\\
    \end{tikzcd}
\end{equation}
where $\kappa$ is the Kummer map and $\hat{E}$ is the formal group associated to $E.$ Moreover, it is such that $\log_{\omega_E}$ computes the formal group logarithm of $\Q\otimes\hat{E}(F_k).$

\begin{proposition}\label{A-comp}
    We have
    \begin{equation}
        A(T,\omega_E,k)=\frac{\abs{\O_{F_k}/\m_k}}{\abs{\tilde{E}(\O_{F_k}/\m_k)[p^\infty]}}
    \end{equation}
    where $\m_k$ is the maximal ideal of $F_k.$
\end{proposition}
\begin{proof}
We have $E(F_k)_{/\tor}\rightiso\Cohomology{1}{\f}{F_k}{T}_{/\tor}$ by the Kummer map, and so, by \eqref{Bloch-Kato-Kummer}, we have
\begin{equation}
    A(T,\omega_E,k)=\frac{\Index{\O_{F_k}}{\log(\hat{E}(\m_k^r))}}{\Index{E(F_k)_{/\tor}}{\hat{E}(\m_k^r)}}\cdot\frac{1}{\abs{E(F_k)_\tor}}=\frac{\Index{\O_{F_k}}{\log(\hat{E}(\m_k^r))}}{\Index{E(F_k)}{\hat{E}(\m_k^r)}}
\end{equation}
where $\m_k$ denotes the maximal ideal of $\O_{F_k}.$ For $r$ sufficiently large, this is
\begin{equation}
    A(T,\omega_E,k)=\frac{\Index{\O_{F_k}}{\m_k^r}}{\Index{E(F_k)}{\hat{E}(\m_k^r)}}=\frac{\Index{\O_{F_k}}{\m_k^r}}{\Index{E(F_k)}{\hat{E}(\m_k)}\cdot \Index{\m_k}{\m_k^r}}=\frac{\Index{\O_{F_k}}{\m_k}}{\Index{E(F_k)}{\hat{E}(\m_k)}}
\end{equation}
where the second equality follows from the exact sequences $0\to \hat{E}(\m_k^{i+1})\to\hat{E}(\m_k^i)\to \m_k^{i+1}/\m_k^i\to 0.$ Because of the exact sequence $0\to \hat{E}(\m_k)\to E(F_k)\to \Tilde{E}(\kappa(F_k))\to 0,$ we get
\begin{equation}
    A(T,\omega_E,k)=\frac{\abs{\O_{F_k}/\m_k}}{\abs{\tilde{E}(\O_{F_k}/\m_k)[p^\infty]}}.\qedhere
\end{equation}
\end{proof}

\begin{theorem}\label{formula-index-2}
    Let $\chi$ be such that $s>f_0$ and such that \eqref{chi-nontrivial} holds. Then
    \begin{equation}
        f_\ac(\chi^{-1}(\gamma)-1)\equal\left(\frac{\log_{\omega_E,v}\kappa_1^\chi}{\mathfrak{g}\left(\chi\lvert_{\Gal{F_s}{F_{f_0}}}\right)}\right)^2
        \iff\frac{\Index{\Cohomology{1}{\BK}{K}{T\otimes\chi}}{\O_L\cdot \kappa_1^\chi}^2}{\prod_{w\mid N^+}c_w(W\otimes\chi)}= \abs{\Sha_{\mathrm{BK}}(W\otimes\chi/K)}.
    \end{equation}
\end{theorem}
\begin{proof}
    By \cref{A-comp}, we have $A(T,\omega_E,s)=A(T,\omega_E,s-1)$ since $s>f_0,$ and hence $f(F_s/F_{s-1})=1.$
    
    Since $V$ satisfies $V\iso V^\tau,$ we have
    \begin{equation}
        \Cohomology{0}{}{K_{\conj{v}}}{W\otimes\chi}\iso\Cohomology{0}{}{K_v}{W^\tau\otimes\chi^{-1}}
    \end{equation}
    and by \cref{twist-calculation} this has the same size as $\Cohomology{0}{}{K_v}{W\otimes\chi}.$ Hence
    \begin{equation}
        \abs{\Cohomology{0}{}{K_v}{W\otimes\chi}}=\abs{\Cohomology{0}{}{K_{\conj{v}}}{W\otimes\chi}},
    \end{equation}
    and now the claim follows from \cref{formula-index} for $z=\kappa_1^\chi$ and $\omega=\omega_E.$
\end{proof}

\section{The BDP Main Conjecture}\label{main-section}
In this section, we continue to assume the setting in the introduction. This means that we consider a quadratic imaginary $K$ with $D_K<-4,$ a prime $p\ge5$ with $p\nmid D_K$ and an elliptic curve $E/\Q$ with $(N_E,D_K)=1$ with good ordinary reduction at $p.$ We assume \eqref{res-surj} and the generalized Heegner hypothesis \eqref{Heeg}. We let $T=T_pE$ be the Tate module of $E.$

For this entire section, we will also assume that $p=v\conj{v}$ splits in $K$ with $v$ the place determined by the fixed embedding $\overline{\Q}\hookrightarrow\overline{\Q_p}.$

\begin{specialtheorem}{BDP Main Conjecture}\label{BDP-main}
    Let $\L_v\in\Lambda^{\ur}\defeq\Z^{\ur}\hat{\otimes}\Lambda$ be the BDP $p$-adic $L$-function of \cite[Theorem 2.7]{Castella-Wan}. Then $X_\ac(\A)$ (see \Cref{section4.1}) is $\Lambda$-torsion, and
    \begin{equation}
        \Char(X_\ac(\A))\Lambda^{\ur}=(\L_v)^2.
    \end{equation}
\end{specialtheorem}

\subsection{BDP formulas}
To prove \cref{Theorem-B}, we will use the following explicit reciprocity law of Heegner points.

\begin{theorem}[{\cite[Theorem 4.9]{Castella-Hsieh}}]\label{BDP-form}
    Let $\chi$ be a finite order anticyclotomic character with $\Ker{\chi}=G_{K_n}$ with $n>n_0.$ Then if $k=k(n)=n-n_0+1,$ we have
    \begin{equation}
        \L_v(\chi^{-1})=\mathfrak{g}(\xi^{-1})\xi(p^k)p^{-k}\sum_{\sigma\in\Gal{K[p^k]}{K}}\chi(\sigma^{-1})\cdot\sigma\left(\log_{\hat{E}}(P[p^k])\right)\equal \frac{\log_{\omega_E,v}(\kappa_1^\chi)}{p^{k/2}}
    \end{equation}
    where $\xi=\chi_v^{-1}$ is the idelic $v$-component of $\chi^{-1}$ seen as a Hecke character, and the Gauss sum $\mathfrak{g}(\xi^{-1})\xi(p^k)$ is the Hecke character Gauss sum.
\end{theorem}
\begin{proof}
Even though \cite{Castella-Hsieh} consider a more restrictive Heegner hypothesis, the proof of \cite[Theorem 4.9]{Castella-Hsieh} works for the generalized Heegner hypothesis \eqref{Heeg}, as noted in the proof of \cite[Theorem 5.6]{Castella-Wan}.

The desired equalities follow by taking $r=1$ and $j=0$ in \cite[Theorem 4.9]{Castella-Hsieh}, and using that $\mathfrak{g}(\xi^{-1})\xi(p^k)\equal p^{k/2}.$
\end{proof}
\begin{corollary}\label{BDP-formula}
    Let $\chi$ be an anticyclotomic character of finite order and local (Galois) conductor $p^s.$ Suppose that $s>f_0.$ Then we have
    \begin{equation}
        \L_v(\chi^{-1})\equal \frac{\log_{\omega_E,v}(\kappa_1^\chi)}{\mathfrak{g}(\chi\lvert_{\Gal{F_s}{F_{f_0}}})}
    \end{equation}
\end{corollary}
\begin{proof}
By the definition of $n$ and $s$ and since $s>f_0,$ we have $s-f_0=n-n_0,$ since both are the exponent of $p$ in the size of the inertia group of $\Gal{K_n}{K}.$ Now the claim follows from the above theorem since
\begin{equation}
    \mathfrak{g}(\chi\lvert_{\Gal{F_s}{F_{f_0}}})\equal p^{(s+1-f_0)/2}=p^{(n+1-n_0)/2}.\qedhere
\end{equation}
\end{proof}

\subsection{Theorem B}
Recall that $X\defeq\Hom_{\Z_p}(\Cohomology{1}{\FLambda}{K}{\A},\Q_p/\Z_p)$ is as in \cref{Howard-main} (see \Cref{section2.4}). To prove \cref{Theorem-B}, we will exploit the equivalence between \cref{Howard-main} and \cref{BDP-main} as in \cite[Theorem 4.1]{BCK}.

\begin{proof}[Proof of \cref{Theorem-B}]
For the case that $p\nmid a_p(a_p-1),$ let $z_\chi\defeq\kappa^\chi_1$ and choose $\chi$ such that $n>n_0$ and such that \eqref{chi-nontrivial} holds. These holds true for infinitely many $\chi$ because of \cite[Theorem 1.10]{Cornut-Vatsal}. For the case that $E/K$ has analytic rank $1,$ we let $y_K\defeq\Tr_{K[1]/K}P[1],$ and will also retain the notation above with $\chi$ being the trivial character, so that we can treat both cases at the same time.

Since we are assuming \cref{Howard-main}, by \cite[Theorem 4.1]{BCK} we have that the \cref{BDP-main} also holds. Combined with \cref{BDP-formula} in the first case and with \cite[Theorem 5.2]{BCK} in the second case, it gives us
\begin{equation}
    f_\ac(\chi^{-1}(\gamma)-1)\equal\left(\frac{\log_{\omega_E,v} z_{\chi}}{\mathfrak{g}\left(\chi\lvert_{\Gal{F_s}{F_{f_0}}}\right)}\right)^2\quad\text{and}\quad f_{\ac}(0)\equal \left(\frac{1-a_p+p}{p}\right)\cdot(\log_Ey_K)
\end{equation}
respectively.

Together with \cref{formula-index-2} in the first case and with \cite[Theorem 5.1]{BCK} in the second case, this implies that
\begin{equation}[]
    \Index{\Cohomology{1}{\BK}{K}{T\otimes\chi}}{z_\chi}^2\equal \abs{\Sha_{\mathrm{BK}}(W\otimes\chi/K)}\cdot\prod_{w\mid N^+}c_w(W\otimes\chi).
\end{equation}

By \cref{main-dvr}$(3),$ this equality is equivalent to
\begin{equation}[eq_d(kappa)]
    p^{2\cdot d\left(\kappa^{\chi}\right)}\equal\prod_{w\mid N^+}c_w(W\otimes\chi).
\end{equation}

By \cref{d=0iffdel=0} and \cref{Tam-cong}\footnote{Note that  $c_w(W)=\lvert c_w(E/K)\rvert_p^{-1}$ by \cite[Lemma 9.1]{Skinner-Zhang}.}, such relation implies that
\begin{equation}
    \kappa^{\chi}\text{ is primitive}\iff \eqref{no-Tam}.
\end{equation}

In the second case, the claim follows since $\kappa^{\chi}=\kappa.$ In the first case, by \cref{primitivity-equivalence}, the left hand side is equivalent to $\kappa$ being primitive, so the conclusion of \cref{Theorem-B} follows.
\end{proof}

\begin{appendix}\eqnumbersubsection
    \section{Computations}
    \subsection{Computations for \texorpdfstring{\cref{first-special-value-formula}}{\ref*{first-special-value-formula}}}\label{appendix-section}
Let $z=\loc{v}c.$ We want to prove \eqref{eq_computation-twist}, that is,
\begin{equation}[eq_appendix-goal]
    \Index{\Cohomology{1}{\f}{\Q_p}{T\otimes\chi}_{/\tor}}{\O_L\cdot z}\cdot\abs{\Cohomology{1}{\f}{\Q_p}{T\otimes\chi}_\tor}\isequal\frac{\abs{\O_{F_s}\otimes_{\Z_p[G_s]}\O_L(\chi^{-1})/\log_\omega z}}{\abs{\O_L/p^s}}\cdot\frac{A(T,\omega,s-1)}{A(T,\omega,s)}.
\end{equation}
We let $G_s\defeq\Gal{F_s}{\Q_p}$ and we regard $(\O_{F_s}\otimes_{\Z_p}\O_L(\chi))^{G_{\Q_p}}\subseteq \O_{F_s}\otimes_{\Z_p[G_s]}\O_L(\chi^{-1})$ by the natural injection.

\begin{lemma}\label{twist-calculation}
    Let $M$ be a $\Z_p\grpring{G_{\Q_p}}$-module. Fix a $p^s$-root of unity $\zeta.$ Then we have an isomorphism
    \begin{equation}
        S_{\chi}\colon M^{G_s}[\Tr_{F_s/F_{s-1}}]\rightiso\left(M\otimes_{\Z_p} \Z_p(\chi)\right)^{G_s}
    \end{equation}
    given by $S_{\chi}(c_0)=\sum_{i=0}^{\phi(p^s)-1}\gamma_i(c_0)\otimes\zeta^i$ where $\gamma_i$ are certain elements of $\Z[G_s].$
\end{lemma}
\begin{proof}
Let $\gamma$ be a generator of $G_s$ such that $\chi(\gamma)=\zeta.$ Then every element of $M\otimes_{\Z_p} \Z_p(\chi)$ can be written uniquely as c=$\sum_{i=0}^{\phi(p^s)-1}c_i\otimes\zeta^i.$ For it to be $G_{\Q_p}$-invariant, we need that each $c_i$ is fixed by $G_{F_s}$ and that the expression is fixed by $\gamma$, that is,
\begin{equation}
    \sum_{i=0}^{\phi(p^s)-1}c_i\otimes\zeta^i=\sum_{i=0}^{\phi(p^s)-1}\gamma c_i\otimes\zeta^{i+1}.
\end{equation}
As $1+\zeta^{p^{s-1}}+\cdots+\zeta^{p^{s-1}(p-1)}=0,$ this is the same as
\begin{equation}
    c_{i+1}= \gamma c_i-\left\{\begin{array}{ll} \gamma c_{\phi(p^s)-1}&\text{if }p^{s-1}\mid i+1,\\0&\text{otherwise,}\end{array}\right.\quad\text{ for }0\le i<\phi(p^s)-1,\quad\text{and}\quad c_0=-\gamma c_{\phi(p^s)-1}.
\end{equation}
In particular, $c$ is determined uniquely by $c_0,$ and the possible $c_0$ are such that
\begin{equation}
    \Tr_{F_s/F_{s-1}}c=\left(1+\gamma^{p^{s-1}}+\cdots+\gamma^{p^{s-1}(p-1)}\right)c=0.
\end{equation}
So the map $c_0\mapsto c$ is the isomorphism we want.
\end{proof}

\begin{lemma}\label{ordinary-lemma}
    We have that $\Cohomology{1}{}{F_k}{V^+}\rightiso\Cohomology{1}{\f}{F_k}{V}$ for all $k\ge0.$ We also have that $\Cohomology{1}{}{F_k}{T^+}\rightiso\Cohomology{1}{\f}{F_k}{T}$ for all $k\ge0$ and that $\Cohomology{0}{}{\Q_p}{W^-\otimes\chi}=0.$
\end{lemma}
\begin{proof}
    We have $\ddr{K}{V^+}\subseteq\ddr{K}{V}.$ Since we are assuming $\lambda_1\le 0<1\le\lambda_2,$ we have $\Fil{0}{\ddr{F_k}{V^+}}=0,$ and so $\frac{\ddr{F_k}{V^+}}{\Fil{0}{\ddr{F_k}{V^+}}}=\frac{\ddr{F_k}{V}}{\Fil{0}{\ddr{F_k}{V}}},$ which means that $\Cohomology{1}{\f}{F_k}{V^+}=\Cohomology{1}{\f}{F_k}{V}.$ Since we are assuming $V^{G_{F_k}}=(V^*(1))^{G_{F_k}}=0,$ local duality and the local Euler characteristic formula gives us that $\Cohomology{1}{}{F_k}{V^+}$ has dimension $\Index{F_k}{\Q_p}$ over $L_0.$ Since the same is true for $\Cohomology{1}{\f}{F_k}{V^+},$ we conclude that
    \begin{equation}
        \Cohomology{1}{\f}{F_k}{V}=\Cohomology{1}{\f}{F_k}{V^+}=\Cohomology{1}{}{F_k}{V^+}.
    \end{equation}
    
    We now note that from \eqref{not-anomalous}, we also obtain $\Cohomology{0}{}{F_k}{W^-}=0.$ Indeed, suppose there exist a nonzero $w\in\Cohomology{0}{}{F_k}{W^-}.$ We may then take $w$ such that $\m_{L_0}w=0.$ Denote by $\psi$ the character of the representation $V^-,$ so that $W^-\iso L_0(\psi)/\O_{L_0}.$ If $\gamma\in\Gal{F_s}{\Q_p},$ we have that $\psi(\gamma^{p^s})w=w,$ that is, that $\psi(\gamma)^{p^s}\equiv1\mod \m_{\O_{L_0}}.$ But this implies $\psi(\gamma)\equiv 1\mod \m_{\O_{L_0}},$ and hence that $\gamma w=w.$ So $w\in (W^-)^{G_{\Q_p}}=0,$ which is a contradiction.
    
    Now consider the following diagram.
    \begin{equation}
    \begin{tikzcd}
        &&&\Cohomology{0}{}{F_k}{W^-}\arrow{d}\\
        \Cohomology{0}{}{F_k}{W^+}\arrow{r}\arrow{d}&\Cohomology{1}{}{F_k}{T^+}\arrow{r}\arrow{d}&\Cohomology{1}{}{F_k}{V^+}\arrow{r}\arrow{d}&\Cohomology{1}{}{F_k}{W^+}\arrow{d}\\
        \Cohomology{0}{}{F_k}{W}\arrow{r}\arrow{d}&\Cohomology{1}{}{F_k}{T}\arrow{r}&\Cohomology{1}{}{F_k}{V}\arrow{r}&\Cohomology{1}{}{F_k}{W}\\
        \Cohomology{0}{}{F_k}{W^-}&&&
    \end{tikzcd}
    \end{equation}
    Using that $\Cohomology{0}{}{F_k}{W^-}=0$ and that $\Cohomology{1}{}{F_k}{V^+}=\Cohomology{1}{\f}{F_k}{V},$ a diagram chasing let us conclude that $\Cohomology{1}{}{F_k}{T^+}=\Cohomology{1}{\f}{F_k}{T}.$
    
    For the last claim, we note, since $\chi\equiv1\mod\m_L,$ that we have $\Cohomology{0}{}{\Q_p}{W^-\otimes\chi}[\m_L]=\left(\Cohomology{0}{}{\Q_p}{W^-}\otimes_{\Z_p}\Z_p[\chi]\right)[\m_L],$ which is $0$ by \eqref{not-anomalous}.
\end{proof}

\begin{corollary}\label{twist-compatibility}
    We have the following isomorphism
    \begin{equation}
        \Cohomology{1}{\f}{\Q_p}{T\otimes\chi}\rightiso\Cohomology{1}{\f}{F_s}{T\otimes\chi}^{G_s}=\left(\Cohomology{1}{\f}{F_s}{T}\otimes_{\Z_p}\Z_p(\chi)\right)^{G_s},
    \end{equation}
    given by restriction.
\end{corollary}
\begin{proof}
    Because of \eqref{no-inv}, we have an isomorphism
    \begin{equation}
        \Cohomology{1}{}{\Q_p}{T^+\otimes\chi}\rightiso\Cohomology{1}{}{F_s}{T^+\otimes\chi}^{G_s}=\left(\Cohomology{1}{}{F_s}{T^+}\otimes_{\Z_p}\Z_p(\chi)\right)^{G_s}.
    \end{equation}
    Now the claim follows once we note that \cref{ordinary-lemma} also works for $T\otimes\chi.$ Indeed, \cref{ordinary-lemma} implies that $V\otimes\chi$ is also not anomalous, and with that, the proof carries through.
\end{proof}

\begin{corollary}\label{C(T)=1}
    For all $k\ge1$ we have a surjection
    \begin{equation}
        \Cohomology{1}{\f}{F_k}{T}\xtwoheadrightarrow{\Tr}\Cohomology{1}{\f}{F_{k-1}}{T}.
    \end{equation}
    \end{corollary}
\begin{proof}
    By \cref{ordinary-lemma}, we have $\Cohomology{1}{\f}{F_k}{T}=\Cohomology{1}{}{F_k}{T^+}$ for all $k.$ By Shapiro's Lemma, the map 
    \begin{equation}
        \Tr\colon \Cohomology{1}{\f}{F_k}{T}\to\Cohomology{1}{\f}{F_{k-1}}{T}
    \end{equation}
    sits in the long exact sequence induced by
    \begin{equation}
        0\to T^+\otimes I_G\to T^+\otimes\Z[G]\xrightarrow{\mathrm{aug}}T^+\to 0.
    \end{equation}
    So the cokernel of $\mathrm{\Tr}$ lies inside
    \begin{equation}
        \Cohomology{2}{}{F_{k-1}}{T^+\otimes I_G}.
    \end{equation}
    It suffices to prove this is $0.$ By local duality and \eqref{tau-dual}, this is dual to the conjugate of
    \begin{equation}
        \Cohomology{0}{}{F_{k-1}}{\Hom_\Z(I_G,W^-)}=\Hom_G(I_G,(W^-)^{G_{F_k}}),
    \end{equation}
    and this is $0$ since $\Cohomology{0}{}{F_k}{W^-}=0$ as in the proof of \cref{ordinary-lemma}.
\end{proof}

Combining \cref{twist-calculation} and \cref{twist-compatibility}, we have
\begin{equation}
    S_{\chi}\colon\Cohomology{1}{\f}{F_s}{T}[\Tr]\rightiso\left(\Cohomology{1}{\f}{F_s}{T}\otimes\Z_p(\chi)\right)^{G_s}\iso\Cohomology{1}{\f}{\Q_p}{T\otimes\chi}.
\end{equation}

For simplicity, we will denote $\Tr=\Tr_{F_s/F_{s-1}}$ and $G=\Gal{F_s}{F_{s-1}}.$ Consider the following diagram.
\begin{equation}
    \begin{tikzcd}
    \frac{\ddr{F_s}{V}}{\Fil{0}{\ddr{F_s}{V}}}[\Tr]&\arrow{l}[swap]{\log}\Cohomology{1}{\f}{F_s}{T}[\Tr]\arrow{r}{S_{\chi}}[swap]{\sim}&\Cohomology{1}{\f}{\Q_p}{T\otimes\chi}\arrow{r}{\log}&\frac{\ddr{}{V\otimes \chi}}{\Fil{0}{\ddr{}{V\otimes\chi}}}\\[-20pt]
    \otimes&&&\otimes\\[-20pt]
    \Fil{1}{\ddr{F_s}{V^\vee}}^{G_{\Q_p}}\arrow[equal]{rrr}\arrow{d}&&&\Fil{1}{\ddr{}{V^\vee}}\arrow{d}\\
    (F_s\otimes_FL_0)[\Tr]\arrow{rrr}{S_\chi}[swap]{\sim}&&&\ddr{}{L(\chi)}
    \end{tikzcd}
\end{equation}

This commutativity means that we can write
\begin{equation}
\begin{split}
    \Index{\Cohomology{1}{\f}{\Q_p}{T\otimes\chi}_{/\tor}}{z}&=\Index{\log_\omega\left(\Cohomology{1}{\f}{\Q_p}{T\otimes\chi}\right)}{\log_\omega z}\\
    &=\Index{S_\chi\log_\omega\left(\Cohomology{1}{\f}{F_s}{T}[\Tr]\right)}{\log_\omega z}\\
    &=\frac{\Index{S_\chi\left(\left(\O_{F_s}\otimes\O_{L_0}\right)[\Tr]\right)}{\log_\omega z}}{\Index{\left(\O_{F_s}\otimes\O_{L_0}\right)[\Tr]}{\log_\omega\left(\Cohomology{1}{\f}{F_s}{T}[\Tr]\right)}}.
\end{split}
\end{equation}

Denote $(\O_{F_s})^\chi\defeq(\O_{F_s}\otimes_{\Z_p}\O_L(\chi))^{G_s}$ and $(\O_{F_s})_\chi\defeq \O_{F_s}\otimes_{\Z_p[G_s]}\O_L(\chi^{-1}).$ By \cref{twist-calculation}, we then have $S_\chi\left(\left(\O_{F_s}\otimes\O_{L_0}\right)[\Tr]\right)=(\O_{F_s})^{\chi},$ and so we can write the above as
\begin{equation}[eq_appendix-partial-progress]
    \Index{\Cohomology{1}{\f}{\Q_p}{T\otimes\chi}_{/\tor}}{z}=\frac{\Index{(\O_{F_s})_\chi}{\log_\omega z}}{\Index{(\O_{F_s})_\chi}{(\O_{F_s})^\chi}}\cdot\frac{1}{\Index{(\O_{F_s}\otimes\O_{L_0})[\Tr]}{\log_\omega\left(\Cohomology{1}{\f}{F_s}{T}[\Tr]\right)}}.
\end{equation}

\begin{proposition}\label{comp1}
We have
\begin{equation}
    \Index{(\O_{F_s})_\chi}{(\O_{F_s})^\chi}=\frac{\abs{\O_L/p^s}}{\Index{\O_{F_{s-1}}}{\Tr(\O_{F_s})}^{\Index{L_0}{\Q_p}}}.
\end{equation}
\end{proposition}
\begin{proof}
We have natural maps
\begin{equation}
    (\O_{F_s})^{\chi}\xrightarrow{\alpha}(\O_{F_s})_{\chi}\xrightarrow{\beta}\frac{1}{p^s}(\O_{F_s})^{\chi}
\end{equation}
where $\alpha$ is the natural inclusion and $\beta(a\otimes b)=\frac{1}{p^s}\sum_{g\in \Gal{F_s}{F}}ga\otimes\chi^{-1}(g)b.$ Since $\O_{F_s}$ is torsion-free, $\alpha$ and $\beta$ are inverses if we extend scalars and so $\abs{\Coker{\alpha}}=(\abs{\O_L/p^s})/\abs{\Coker{\beta}}.$

If $\gamma$ is such that $\chi(\gamma)=\zeta,$ then $\beta(a\otimes 1)=S_\chi(a-\gamma a),$ and so we may conclude that $\Im{\beta}=\frac{1}{p^s}S_\chi(I_G\O_{F_s})\otimes_{\Z_p}\O_{L_0},$ and hence, by \cref{twist-calculation}, that $\abs{\Coker{\beta}}=\Index{\O_{F_s}[\Tr]}{I_G\O_{F_s}}^{\Index{L_0}{\Q_p}}.$ Since $G$ is cyclic, this is the same as $\Index{O_{F_{s-1}}}{\Tr(\O_{F_s})}^{\Index{L_0}{\Q_p}}.$
\end{proof}

\begin{proposition}\label{comp2}
We have
    \begin{equation}
        \Index{\log_\omega(\Cohomology{1}{\f}{F_s}{T})[\Tr]}{\log_\omega(\Cohomology{1}{\f}{F_s}{T}[\Tr])}=\frac{\abs{\Cohomology{1}{\f}{F_{s-1}}{T}_\tor}}{\abs{\Tr\left(\Cohomology{1}{\f}{F_{s}}{T}_\tor\right)}}.
    \end{equation}
\end{proposition}
\begin{proof}
By the injectivity of $\log_\omega$ up to torsion, the index in the left hand side is
\begin{equation}
    \Index{(\Cohomology{1}{\f}{F_s}{T}_{/\tor})[\Tr]}{\Cohomology{1}{\f}{F_s}{T}[\Tr]_{/\tor}},
\end{equation}
and now the claim follows from a snake lemma in the following diagram together with \cref{C(T)=1}.
\begin{equation}
    \begin{tikzcd}
    0\arrow{r}&\Cohomology{1}{\f}{F_s}{T}_{\tor}\arrow{r}\arrow{d}{\Tr}&\Cohomology{1}{\f}{F_s}{T}\arrow{r}\arrow{d}{\Tr}&\Cohomology{1}{\f}{F_s}{T}_{/\tor}\arrow{r}\arrow{d}{\Tr}&0\\
    0\arrow{r}&\Cohomology{1}{\f}{F_{s-1}}{T}_\tor\arrow{r}&\Cohomology{1}{\f}{F_{s-1}}{T}\arrow{r}&\Cohomology{1}{\f}{F_{s-1}}{T}_{/\tor}\arrow{r}&0
    \end{tikzcd}\qedhere
\end{equation}
\end{proof}

We recall that
\begin{equation}
    A(T,\omega,k)\defeq\frac{\Index{\O_{F_k}\otimes\O_{L_0}}{\log_\omega\left(\Cohomology{1}{\f}{F_k}{T}\right)}}{\abs{\Cohomology{1}{\f}{F_k}{T}}_\tor},
\end{equation}
and we will also denote
\begin{equation}
    A'(T,\omega,k)\defeq\Index{\O_{F_k}\otimes\O_{L_0}}{\log_\omega\left(\Cohomology{1}{\f}{F_k}{T}\right)}.
\end{equation}

\begin{proposition}\label{comp3}
We have
\begin{equation}
    \Index{(\O_{F_s}\otimes\O_{L_0})[\Tr]}{\log_\omega(\Cohomology{1}{\f}{F_s}{T})[\Tr]}=\Index{\O_{F_{s-1}}}{\Tr(\O_{F_s})}^{\Index{L_0}{F}}\cdot\frac{A'(T,\omega,s)}{A'(T,\omega,s-1)}.
\end{equation}
\end{proposition}
\begin{proof}
Let $N$ be a sufficiently large integer such that $\log_\omega(\Cohomology{1}{\f}{F_s}{T})\subseteq \frac{1}{p^N}\O_{F_s}\otimes\O_{L_0},$ so that
\begin{equation}
\begin{split}
    \Index{(\O_{F_s}\otimes\O_{L_0})[\Tr]}{\log_\omega(\Cohomology{1}{\f}{F_s}{T})[\Tr]}&=\frac{\Index{\frac{1}{p^N}(\O_{F_s}\otimes\O_{L_0})[\Tr]}{\log_\omega(\Cohomology{1}{\f}{F_s}{T})[\Tr]}}{\Index{\frac{1}{p^N}\O_{F_s}[\Tr]}{ \O_{F_s}[\Tr]}^{\Index{L_0}{F}}}\\
    &=\frac{\Index{\frac{1}{p^N}(\O_{F_s}\otimes\O_{L_0})[\Tr]}{\log_\omega(\Cohomology{1}{\f}{F_s}{T}_{/\tor}[\Tr])}}{p^{N\Index{L}{F}}}.
\end{split}
\end{equation}

Now consider the following commutative diagram
\begin{equation}
    \begin{tikzcd}
    0\arrow{r}& \Cohomology{1}{\f}{F_s}{T}_{/\tor}\arrow{r}{\log_\omega}\arrow{d}{\Tr}&\frac{1}{p^N}\O_{F_s}\otimes\O_{L_0}\arrow{r}\arrow{d}{\Tr}&A_s\arrow{r}\arrow{d}{\Tr}&0\\
    0\arrow{r}&\Cohomology{1}{\f}{F_{s-1}}{T}_{/\tor}\arrow{r}{\log_\omega}&\frac{1}{p^N}\O_{F_{s-1}}\otimes\O_{L_0}\arrow{r}&A_{s-1}\arrow{r}&0
    \end{tikzcd}
\end{equation}

The modules $A_s$ and $A_{s-1}$ are finite, so a Snake Lemma gives us that
\begin{equation}
    \Index{\frac{1}{p^N}(\O_{F_s}\otimes\O_{L_0})[\Tr]}{\log_\omega(\Cohomology{1}{\f}{F_s}{T}_{/\tor}[\Tr])}=\frac{\abs{A_{s}}}{\abs{A_{s-1}}}\cdot \Index{\O_{F_{s-1}}}{\Tr(\O_{F_s})}^{\Index{L_0}{F}}.
\end{equation}

Since $\abs{A_k}=p^{Np^k\Index{L_0}{F}}\cdot A'(T,\omega,k),$ we have
\begin{equation}
    \frac{\abs{A_s}}{\abs{A_{s-1}}}=p^{N\phi(p^s)\Index{L_0}{F}}\cdot\frac{A'(T,\omega,s)}{A'(T,\omega,s-1)}=p^{N\Index{L}{F}}\cdot\frac{A'(T,\omega,s)}{A'(T,\omega,s-1)}.
\end{equation}

Combining the equations above give the desired result.
\end{proof}

Combining \eqref{eq_appendix-partial-progress} and \cref{comp1,comp2,comp3}, we get
\begin{equation}
    \Index{\Cohomology{1}{\f}{\Q_p}{T\otimes\chi}_{/\tor}}{\O_L\cdot z}=\frac{\Index{(\O_{F_s})_\chi}{\log_\omega z}}{\abs{\O_L/p^s}}\cdot\frac{A'(T,\omega,s-1)}{A'(T,\omega,s)}\cdot\frac{\abs{\Tr\left(\Cohomology{1}{\f}{F_s}{T}_\tor\right)}}{\abs{\Cohomology{1}{\f}{F_{s-1}}{T}_\tor}}.
\end{equation}
Since $\Cohomology{1}{\f}{\Q_p}{T\otimes\chi}_\tor\iso\Cohomology{1}{\f}{F_s}{T}[\Tr]_\tor$ by \cref{twist-calculation}, this is the same as
\begin{equation}
    \Index{\Cohomology{1}{\f}{\Q_p}{T\otimes\chi}_{/\tor}}{\O_L\cdot z}\cdot\abs{\Cohomology{1}{\f}{\Q_p}{T\otimes\chi}_\tor}=\frac{\Index{(\O_{F_s})_\chi}{\log_\omega z}}{\abs{\O_L/p^s}}\cdot\frac{A(T,\omega,s-1)}{A(T,\omega,s)},
\end{equation}
which is \eqref{eq_appendix-goal}.
    \subsection{Computations for \texorpdfstring{\cref{formula-index}}{\ref*{formula-index}}}\label{app-sec2}
Let $F_\infty^\ur$ and $F_\infty^{\mathrm{cyc}}$ be respectively the unramified and cyclotomic $\Z_p$-extensions of $\Q_p.$ By local class field theory, there is a unique $\Z_p\times\Z_p$ extension of $\Q_p,$ which we will call $M_\infty.$ Then we have
\begin{equation}
    M_\infty=F_\infty^\ur F_\infty=F_\infty^\ur F_\infty^{\mathrm{cyc}}.
\end{equation}
Let $f=p^{f_0}=f(F_\infty/F).$ Then, by checking the ramification indexes, we have
\begin{equation}
    F_\infty^\ur F_s=F_\infty^\ur F_{s-f_0}^{\mathrm{cyc}}\eqdef M_s.
\end{equation}

We may picture these extensions in the following diagram, where the labels are lifts of generators of the Galois groups to $\Gal{M_s}{\Q_p}.$

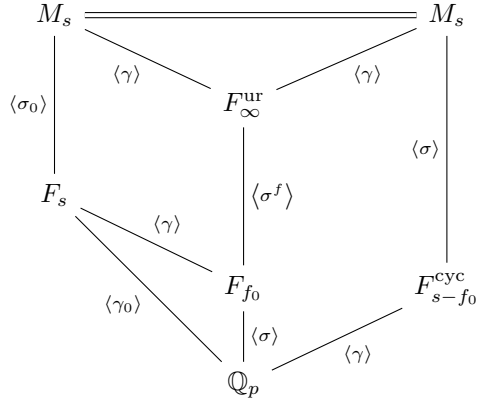
\begin{equation}
    \begin{tikzcd}
    M_s\arrow[equal]{rrrr}\arrow[swap, dash]{rrd}{\left<\gamma\right>}&&&&M_s\arrow[dash]{dll}{\left<\gamma\right>}\\
    &&F_\infty^\ur&&\\
    F_s\arrow[dash]{uu}{\left<\sigma_0\right>}&&&&\\
    &&F_{f_0}\arrow[swap, dash]{ull}{\left<\gamma\right>}\arrow[swap, dash]{uu}{\left<\sigma^f\right>}&&F_{s-f_0}^{\mathrm{cyc}}\arrow[dash]{uuu}{\left<\sigma\right>}\\
    &&\Q_p\arrow[dash]{lluu}{\left<\gamma_0\right>}\arrow[swap, dash]{u}{\left<\sigma\right>}\arrow[swap, dash]{urr}{\left<\gamma\right>}&&
    \end{tikzcd}
\end{equation}
We may choose $\sigma_0=(\gamma^a,\sigma^f)$ for some $a\in\Z_p.$ If $f>1,$ then we can choose $\gamma_0=(\gamma^b,\sigma^c)$ such that $\gamma_0^f\equiv \gamma,$ that is, $fb-ac\equiv1\mod p^s.$

With the notation of \Cref{appendix-section}, we want to understand the image of $(\O_{F_s})_{\chi}\O_\infty^\ur$ in $L F_\infty^\ur,$ via the map
\begin{equation}
    S\colon a\otimes b\mapsto \frac{1}{p^s}\sum_{g\in \Gal{F_s}{\Q_p}}g(a)\chi(g)b\in F_sL\subseteq M_sL=L F_\infty^\ur.
\end{equation}
Note that this is the map such that the following diagram is commutative.
\begin{equation}
\begin{tikzcd}
    (\O_{F_s})^\chi\arrow[hook]{d}\arrow[hook]{dr}&\\
    (\O_{F_s})_\chi\arrow{r}{S}&\O_{F_s}\O_L
\end{tikzcd}
\end{equation}

\begin{theorem}
    We have
    \begin{equation}
        S((\O_{F_s})_\chi\O_\infty^\ur)=\frac{\mathfrak{g}\left(\chi\lvert_{\Gal{F_s}{F_{f_0}}}\right)}{p^s}\O_L\O_\infty^\ur.
    \end{equation}
\end{theorem}
\begin{proof}

Let $\zeta$ be a primitive $p^{s-f_0+1}$-th root of unity. Let $g$ be a primitive root of $\Z/p^{s-f_0+1}\Z.$ We let $\zeta_i\defeq\sum_{j=1}^{p-1}\zeta^{g^{i+p^{s-f_0}f}}.$ be a normal basis of $\O_{s-f_0}^{\mathrm{cyc}}.$ We may choose $g$ such that $\gamma(\zeta_i)=\zeta_{i+1}.$ For $\alpha\in \O_{F_s},$ there is a unique decomposition $\alpha=\sum_{l=1}^{p^s/f}\alpha_l\zeta_l$ for $\alpha_l\in \O_\infty^\ur.$

Suppose first that $f=1.$ Then we have
\begin{equation}
    S(\alpha)=\frac{1}{p^s}\sum_{i=1}^{p^s}\gamma^i\left(\sum_{l=1}^{p^s}\alpha_l\zeta_l\right)\chi(\gamma^i)=\frac{1}{p^s}\sum_{l=1}^{p^s}\alpha_l\sum_{i=1}^{p^s}\zeta_{l+i}\chi(\gamma^i)=\frac{\mathfrak{g}(\chi\lvert_{\Gal{F_s}{\Q_p}})}{p^s}\cdot\sum_{l=1}^{p^s}\alpha_l\chi^{-1}(\gamma^l).
\end{equation}
Since $\alpha\in F_s,$ we must have $\sigma_0(\alpha)=\alpha,$ that is, $\sigma(\alpha_l)=\alpha_{l+a}.$ We have that $\sum_{l=1}^{p^s}\alpha_l\chi^{-1}(\gamma^l)\equiv \sum_{l=1}^{p^s}\alpha_l\mod\m_L,$ so if $p^r\| a,$ we have
\begin{equation}
    \sum_{l=1}^{p^s}\alpha_l\chi^{-1}(\gamma^l)\equiv \sum_{l=1}^{p^r}\Tr_{F_{s-r}^\ur/\Q_p}(\alpha_l)\mod\m_L,
\end{equation}
and $\alpha$ is determined by $\alpha_l\in \O_{s-r}^\ur$ for $1\le l\le p^r.$ Since trace is surjective on the ring of integers for unramified extensions, this means that the image we want is $\frac{\mathfrak{g}(\chi\lvert_{\Gal{F_s}{\Q_p}})}{p^s}\O_L\O_\infty^\ur.$

Now suppose $f>1.$ Then, by $p^s\mid fb-ac-1,$ we have $p\nmid a.$ This means that we may write $\alpha_l=\sigma^{fl/a}(\alpha_0).$ So
\begin{equation}
    S(\alpha)=\frac{1}{p^s}\sum_{i=1}^{p^s}\gamma_0^i\left(\sum_{l=1}^{p^s/f}\sigma^{fl/a}(\alpha_0)\zeta_l\right)\chi(\gamma_0^i)=\frac{1}{p^s}\sum_{i=1}^{p^s}\sum_{l=1}^{p^s/f}\sigma^{fl/a+ci}(\alpha_0)\zeta_{l+bi}\chi(\gamma_0^i).
\end{equation}
Changing $l\mapsto l-bi$ and then $i\mapsto fl-i,$ and using that $p^s\mid fb-ac-1,$ we have
\begin{equation}
    S(\alpha)=\frac{1}{p^s}\sum_{i=1}^{p^s}\sigma^{i/a}(\alpha_0)\chi^{-1}(\gamma_0^i)\sum_{l=1}^{p^s/f}\zeta_l\chi(\gamma_0^{fl})=\frac{\mathfrak{g}\left(\chi\lvert_{\Gal{F_s}{F_{f_0}}}\right)}{p^s}\sum_{i=1}^{p^s}\sigma^{i/a}(\alpha_0)\chi^{-1}(\gamma_0^i).
\end{equation}
As before, we have $\sum_{i=1}^{p^s}\sigma^{i/a}(\alpha_0)\chi^{-1}(\gamma_0^i)\equiv \sum_{i=1}^{p^s}\sigma^{i/a}(\alpha_0)\equiv\Tr_{F_s^\ur/F_0}(\alpha_0)\mod \m_L.$ So the claim follows.
\end{proof}

\begin{corollary}
Let $c\in\Cohomology{1}{\mathcal{F}}{K}{T\otimes\chi}$ be any class with non-torsion localization at $v.$ We have
\begin{equation}
    \frac{\abs{(\O_{F_s})_\chi/\log_\omega(\loc{v}c)}}{\abs{\O/p^s}}=\abs{\O_L/\frac{u\cdot\log_{\omega,v} c}{\mathfrak{g}\left(\chi\lvert_{\Gal{F_s}{F_{f_0}}}\right)}}.
\end{equation}
for some $u\in\units{(\Z_p^\ur)}.$
\end{corollary}
\end{appendix}

\newpage
\begingroup
\setstretch{2.5}
\bibliographystyle{alpha}
\bibliography{references}
\endgroup

\end{document}